\documentclass[12pt]{article}
\usepackage{amsmath,amscd,amsbsy,amssymb,latexsym,url,bm,amsthm}
\usepackage[vlined,boxed,commentsnumbered,linesnumbered,ruled]{algorithm2e}
\usepackage{epsfig,graphicx,subfigure}
\usepackage{enumitem,balance,mathtools}
\usepackage{wrapfig}
\usepackage{mathrsfs, euscript}
\usepackage[usenames]{xcolor}
\usepackage{hyperref}
\usepackage{epstopdf}
\usepackage{float}
\usepackage{comment}
\author{Pengyu Chen\thanks{School of Mathematical Sciences, Shanghai Jiao Tong University, Shanghai, China \href{mailto:fwang256@sjtu.edu.cn}{\texttt{chenpy653@sjtu.edu.cn}}} \and Zeren Zhang\thanks{School of Mathematical Sciences, Shanghai Jiao Tong University, Shanghai, China \href{mailto:zhangzr0018@sjtu.edu.cn}{\texttt{zhangzr0018@sjtu.edu.cn}}}}
\title{The stability threshold for Boussinesq equations around stratified Couette flow with unequal
viscosity and thermal diffusivity on $\mathbb{T}\times\mathbb{R}^2$ }
\date{}
\allowdisplaybreaks[4]
\def\d{\mathrm{d}}

\def\'{^{\prime}}

\newcommand{\wtb}{\langle\partial_Z\rangle }

\newcommand{\nm}[2]{\left\| #1 \right\| _{#2} }

\newtheorem{theorem}{Theorem}[section]
\newtheorem{lemma}{Lemma}[section]
\newtheorem{proposition}{Proposition}[section]

\newtheorem{remark}{Remark}[section]

\numberwithin{equation}{section}

\begin{document}
    \maketitle
    \begin{abstract}
		We establish a stability threshold for the 3D Boussinesq equations around Couette flow $(y,0,0)$ with linearly stratified temperature profile $1+y$ on $\mathbb{T}\times\mathbb{R}^2$, with distinct viscosity $\kappa$ and thermal diffusivity $\mu$ and arbitrary Brunt-V\"ais\"al\"a frequency $\beta>0$. Assume that $\mu$ and $\kappa$ are comparable, and set $\nu=\min\{\kappa,\mu\}$. For sufficiently small \(a_0>0\) and \(s\geq4\), we prove that if the initial perturbation satisfies $\|(u_{in},\theta_{in})\|_{H^s}+\|(u_{0,in},\theta_{0,in})\|_{W^{s,1}}\lesssim\nu^{\frac{2}{3}} |\ln{\nu}|^{-2-2a_0}$, then the corresponding solution exists globally in time. We also establish Strichartz-type estimates for zero mode. Moreover, our estimates further allow for a nonlinear amplification of order $\nu^{-1/6}$ in the high-regularity norm of the zero mode of the first component $u^1_0$.
	\end{abstract}
	\section{Introduction}
		In this paper, we consider the 3D Boussinesq equations on $\mathbb{T}\times\mathbb{R}^2$:
	\begin{equation*}
		\left\{\begin{array}{l}
			\partial_t \tilde{u}+\tilde{u} \cdot \nabla \tilde{u}+\nabla \tilde{p}=\kappa \Delta \tilde{u}+(0,\mathfrak{g}\theta,0)^T, \\
			\partial_t \tilde{\theta}+\tilde{u} \cdot \nabla \tilde{\theta}=\mu \Delta \tilde{\theta}, \\
			\nabla \cdot \tilde{u}=0 ,
		\end{array}\right.
	\end{equation*}
	where $\tilde{u}$ is the velocity, $\tilde{\theta}$ is the temperature, $\tilde{p}$ is the pressure, $\kappa$ is the viscosity coefficient, $\mu$ is the thermal diffusivity, and  $\mathfrak{g}>0$ is the constant of gravity. A stationary states
	\begin{equation*}
		u^s=(y,0,0)^T,\quad\partial_yp^s=\mathfrak{g}(1+\alpha y),\quad\theta^s=1+\alpha y,
	\end{equation*}
	where the constant $\alpha>0$ is the Richardson number.
	We introduce the perturbations of the Couette flow and stratified temperature:
	\begin{equation*}
		u=\tilde{u}-(y,0,0)^{T},\quad\sqrt{\alpha/\mathfrak{g}}\ \theta=(1+\alpha y)-\tilde{\theta},
	\end{equation*}
	and then $(u,\theta)$ satisfies
	\begin{equation}\label{eq-main}
		\left\{
		\begin{aligned}
			&\partial_t u+y\partial_x u+(u^2,0,0)^T+(u\cdot\nabla)u+\nabla p^L+\nabla p^{NL}=\kappa \Delta u-\beta(0,\theta,0)^T,\\
			&\partial_t \theta+y\partial_x \theta-\beta u^2+(u\cdot\nabla)\theta=\mu\Delta\theta,\\
			&\Delta p^L=-2\partial_x u^2-\beta\partial_y \theta,\quad\Delta p^{NL}=-\partial_i u^j\partial_j u^i,\\
			&(u,\theta)|_{t=0}=(u_{in},\theta_{in}),
		\end{aligned}
		\right.
	\end{equation}
	where $\beta=\sqrt{\alpha\mathfrak{g}}>0$ is the Brunt-V\"ais\"al\"a frequency. 
    Define
    \begin{equation*}
        \nu:=\min\{\kappa,\mu\}.
    \end{equation*}To understand the underlying transition mechanisms of the above system, we are concerned with the stability threshold problem in the Sobolev spaces and formulate it as \cite{BGM17}:
	
	\textbf{Stability threshold}: Given $N\geq0$, determine $\gamma\in\mathbb{R}$ such that
	\begin{align*}
		&\|(u_{in},\theta_{in})\|_{H^N}\ll\nu^{\gamma}\ \Rightarrow\ \mathrm{stability},\\
		&\|(u_{in},\theta_{in})\|_{H^N}\gg\nu^{\gamma}\ \Rightarrow\ \mathrm{possible \ instability}.
	\end{align*}
	
	Quantifying stability thresholds represents an active research frontier originating from fluid dynamics studies of the Navier-Stokes equations.	
	A seminal contribution by Bedrossian, Germain, and Masmoudi \cite{BGM17} established the threshold $\gamma = \frac{3}{2}$ for 3D perturbations of Couette flow on $\mathbb{T} \times \mathbb{R} \times \mathbb{T}$ in Sobolev spaces. Subsequent refinements by Wei and Zhang \cite{WZ21} improved this to $\gamma = 1$ in $H^2$. In two dimensions, the absence of the lift-up effect leads to substantially lower thresholds. For perturbation on $\mathbb{T} \times \mathbb{R}$, Bedrossian, Vicol, and Wang \cite{BVW18} established a threshold $\gamma = \frac{1}{2}$, with later work \cite{MZ22,WZ23} sharpening this to $\gamma = \frac{1}{3}$. Related problem in Gevrey classes has also been studied, with results on $\mathbb{T}\times\mathbb{R}$ and $\mathbb{T}\times\mathbb{R}\times\mathbb{T}$ given in \cite{BMV16,LMZ25} and \cite{BGM20,BGM22}, respectively. The corresponding problem in bounded domains has led to a number of further developments \cite{CWZ24,CLWZ20,BHIW25,WZ26,BHIW24,BHILW25a}. For broader perspectives on general shear flows, we refer to \cite{LWZ20,C23,AB25,CJWZ25,LLZ25,LSZ26,CEW20}.
	
	The Boussinesq equations provide a natural framework for investigating the interplay between mixing and buoyancy. In two dimensions, the stability of Couette flow has been studied extensively. For Richardson number $\alpha>\frac{1}{4}$, a threshold $\gamma=\frac{1}{2}$ \cite{ZZ23} was established on $\mathbb{T}\times\mathbb{R}$ in Sobolev spaces, and later improved to $\frac{1}{3}$ in \cite{RW26,K25}. In particularly, the threshold in \cite{RW26} was estabilish for $\kappa+\mu\lesssim\sqrt{\kappa\mu}$. For perturbations on $\mathbb{R}^2$, a slight larger threshold $\gamma=\frac{1}{2}+$ \cite{A25} was also obtained. Stability studies in Gevrey class can be found in \cite{MSZ22,BBCD2023}. In the unstratified case $\alpha=0$, stability results have also been established on $\mathbb{T}\times\mathbb{R}$, $\mathbb{T}\times[-1,1]$, and $\mathbb{R}^2$; see \cite{DWZ21,NZ24,ZZi23}, \cite{MZZ23,LWZ25}, and \cite{CWY25}, respectively. In three dimensions, stratification gives rise to internal gravity waves and associated dispersive effects. A key breakthrough in this direction was achieved by Coti Zelati, Del Zotto, and Widmayer \cite{CD23,CDW24}, who established the improved stability thresholds on $\mathbb{T}\times\mathbb{R}\times\mathbb{T}$ in Sobolev spaces: namely $\gamma=\frac{11}{12}$ for Brunt-V\"ais\"al\"a frequency $\beta>\frac{1}{2}$, and $\gamma=\frac{8}{9}$ for $\beta>\nu^{-\frac{1}{2}}$. 
	Remarkably, these thresholds are quite lower than the threshold $\gamma=1$ for the 3D Navier--Stokes equations, thereby demonstrating the stabilizing effect of stratification.	Subsequent work \cite{CWW25} considered the unstratified case $\alpha=0$, obtaining the threshold $\|u_{in}\|_{H^2}\lesssim\nu, \|\theta_{in}\|_{H^2}\lesssim\nu^2$. Analogous dispersion--mixing interactions have been explored in the Navier--Stokes--Coriolis system. Recently, Li, Sun, Wang, Wei, and Zhang \cite{LSWZ25} established the thresholds $\gamma=\frac{2}{3}+$ on $\mathbb{T}\times\mathbb{R}^2$ and $\gamma=\frac{5}{6}$ on $\mathbb{T}\times\mathbb{R}\times\mathbb{T}$.
	For more results on this system can be found in \cite{HSX24a,CDW25}. More broadly, the interplay between mixing and additional physical effects has been investigated in various coupled systems, including rotating Boussinesq, magnetohydrodynamic, and compressible flows; see \cite{HLSX26,L20,K25,HLX26,LWZ25} and the references therein.
	
	The purpose of this paper is to establish the threshold $\gamma=\frac{2}{3}+$ for \eqref{eq-main} on $\mathbb{T}\times\mathbb{R}^2$ for $\beta>0$ when $\kappa\approx\mu$. In addition, our result allow for an amplification of order $\nu^{-\frac{1}{6}}$ in the high regularity norm of $u^1_0$. Define
	\begin{equation*}
		f_0=\int_{\mathbb{T}}f(x,y,z)dx,\quad f_{\neq}=f-f_0.
	\end{equation*}
    Our main result is stated as follows:
	\begin{theorem}\label{mainthe}
		Assume that $\beta>0$, $0<\kappa,\mu\leq1$, and $C_1^{-1}\mu\leq\kappa\leq C_1\mu$ for some constant $C_1\geq1$. Set $\nu=\min\{\kappa,\mu\}$ and let $(u_{in},\theta_{in})$ be the initial datum for \eqref{eq-main}. For any $0<a_0\ll1$ and $s\geq4$, there exists $\epsilon_0=\epsilon_0(\beta,s,a_0,C_1)>0$ such that, if it holds that
		\begin{equation}
			\|\langle \nabla_{x,z}\rangle (u_{in},\theta_{in})\|_{H^{s+1}}+\|\wtb(u_{0,in},\theta_{0,in})\|_{W^{s,1}}=\epsilon\leq\epsilon_0\nu^{\frac{2}{3}}|\ln{\nu}|^{-2-2a_0},
		\end{equation}
		then the system \eqref{eq-main} admits a unique global solution. Moreover, the profile $(U,\Theta)(t,x,y,z)=(u,\theta)(t,x+yt,y,z)$ satisfies the following estimates:
		\begin{subequations}\label{est-main}
			\begin{align}
				&\|e^{a_1\nu^{\frac{1}{3}}t}\partial_x\nabla_{x,z}U^1_{\neq}\|_{L^\infty H^s}+\nu^{\frac{1}{6}}\|\partial_x\nabla_{x,z}U^1_{\neq}\|_{L^2H^s}\lesssim|\ln\nu|^{1+a_0}\epsilon,\label{main-u1}\\
				&\|e^{a_1\nu^{\frac{1}{3}}t}\nabla_{x,z}\nabla_LU^2_{\neq}\|_{L^\infty H^s}+\|\partial_x\nabla_{x,z}U^2_{\neq}\|_{L^2 H^s}+\nu^{\frac{1}{6}}\|\nabla_{x,z}\nabla_LU^2_{\neq}\|_{L^\infty H^s}\lesssim\epsilon,\label{main-u2}\\
				&\|e^{a_1\nu^{\frac{1}{3}}t}\nabla_{x,z}^2U^3_{\neq}\|_{L^\infty H^s}+\nu^{\frac{1}{6}}\|\nabla_{x,z}^2U^3_{\neq}\|_{L^2H^s}\lesssim|\ln\nu|^{1+a_0}\epsilon,\label{main-u3}\\
				&\|e^{a_1\nu^{\frac{1}{3}}t}\nabla_{x,z}^2\Theta_{\neq}\|_{L^\infty H^s}+\nu^{\frac{1}{6}}\|\nabla_{x,z}^2\Theta_{\neq}\|_{L^2H^s}\lesssim\epsilon,\label{main-theta}\\
				&\|\wtb U^1_0\|_{L^\infty H^s}+\nu^{\frac{1}{6}}\| U^1_0\|_{L^\infty H^{s+1}}+\|\wtb (U^{2}_0,U^3_0,\Theta_0)\|_{L^\infty H^{s+1}}\lesssim\epsilon,\label{main-u10}\\
				&\nu^{\frac{1}{12}+a_2}\|\wtb(U^2_0,\Theta_0)\|_{L^2W^{s-2,\infty}}+\nu^{a_2}\|\wtb U^3_{0}\|_{L^2W^{s-1}}\lesssim\epsilon.\label{main-u230}
			\end{align}
		\end{subequations}
		Here $a_1,a_2$ are sufficiently small constants, $\nabla_{x,z}=(\partial_x,\partial_z)$, and $\nabla_L=(\partial_x,\partial_y-t\partial_x,\partial_z)$.
	\end{theorem}
	\begin{remark}
		The inviscid damping of $U^2_{\neq}$ is reflected in the uniform $\nu$-bounds in \eqref{main-u2}. 
		The	enhanced dissipation of the nonzero modes is characterized by the $\nu^{-\frac{1}{6}}$ growth of the $L^2$ in time estimates in \eqref{main-u1}--\eqref{main-theta}. Furthermore, the first term of \eqref{main-u10} describes the suppression of the lift-up effect for zero-mode. It is worth noting that  $\|U^1_0\|_{H^{s+1}}$ exhibits a growth of order $\nu^{-\frac{1}{6}}$ due to the nonlinear effects. Finally, \eqref{main-u230} provides Strichartz-type estimates for $(u^2_0,u^3_0,\theta_0)$ arising from the dispersive structure.
	\end{remark}
	\begin{remark}
		In the stable regime $\beta^2>\frac{1}{4}$, the estimates in \eqref{est-main} imply the following inviscid damping and enhanced dissipation estimates:
		\begin{align}\label{est-idl}
			|\ln(\nu)|^{-(1+a_0)}\|(U^1_{\neq},U^3_{\neq})(t)\|_{L^2}+\langle t\rangle \|U^2_{\neq}(t)\|_{L^2}\lesssim e^{-a_0\nu^{\frac{1}{3}}t}\epsilon,
		\end{align}
		However, the following linear estimate has been established in \cite{CD23} :
		\begin{align}\label{est-id}
			\|(U^1_{\neq},U^3_{\neq})(t)\|_{L^2}+\langle t\rangle^{\frac{1}{2}}\|\Theta_{\neq}(t)\|_{L^2}+\langle t\rangle^{\frac{3}{2}} \|U^2_{\neq}(t)\|_{L^2}\lesssim e^{-a_0\nu^{\frac{1}{3}}t}\epsilon.
		\end{align} 
		The loss in the decay rates of $(U^2_{\neq},\Theta_{\neq})$ in \eqref{est-idl} is caused by the nonlinear growth from $u^{3,lo}_{\neq}\partial_z(u^{2}_0,\theta_0)^{hi}$, which must be offset by additional decay, see \eqref{ineq-t41} for details. On the other hand, the logarithmic loss for $(U^1_{\neq},U^3_{\neq})$ is related to the treatment of the high-regularity estimate of $u^1_0$, as explained in the main idea of proof below.
				
		The optimal nonlinear stability estimates \eqref{est-id} can be recovered by combining the good unknown introduced in \cite{CDW24} with the method developed in the present paper, at the expense of a smaller size assumption on the initial data. More precisely, one has the following result.
		\begin{proposition}
			Assume $\beta>1/2$, $0<\mu,\kappa\leq1$ and $C_0\mu\leq\kappa\leq C_0\mu$ for some $C_0\geq1$. Let $(u_{in},\theta_{in})$ be the initial datum of \eqref{eq-main}. For any $0<\delta\ll1$ and $s\geq4$, there exists $\epsilon_0=\epsilon_0(\beta,s,\delta,C_0)>0$ such that if it holds that
			\begin{equation}
				\| (u_{in},\theta_{in})\|_{H^{s+2}}+\|(u_{0,in},\theta_{0,in})\|_{W^{s+1,1}}=\epsilon\leq\epsilon_0\nu^{\frac{3}{4}+\delta},
			\end{equation}
			then the system \eqref{eq-main} admits a unique global solution, and the stability estimate \eqref{est-id} holds.
		\end{proposition}
		  It remains open whether the optimal estimate \eqref{est-id} can be achieved under the condition $\gamma\leq\frac{2}{3}+\delta$, or whether the sharp stability threshold can be further identified.		 
	\end{remark}
    
    \begin{remark}
        For comparison, the linearized two-dimensional problem on $\mathbb{T}\times\mathbb{R}$ exhibits different decay regimes depending on the stratification strength. The estimates established in \cite{YL18} can be summarized as follows:
        \begin{enumerate}
            \item If $0<\beta^2<\frac{1}{4}$, then
            \begin{align*}
             \langle t\rangle^{\frac12}\|U^{x}_{\neq}(t)\|_{L^2}+\langle t\rangle^{\frac32} \|U^y(t)\|_{L^2}+\langle t\rangle^{\frac12}\|\Theta_{\neq}(t)\|_{L^2}\lesssim \langle t\rangle^{\sqrt{\frac14-\beta^2}}\epsilon;
        \end{align*}
        \item If $\beta^2=\frac{1}{4}$, then
        \begin{align*}
            \langle t\rangle^{\frac12}\|U^{x}_{\neq}(t)\|_{L^2}+\langle t\rangle^{\frac32} \|U^y(t)\|_{L^2}+\langle t\rangle^{\frac12}\|\Theta_{\neq}(t)\|_{L^2}\lesssim \langle\log\langle t\rangle\rangle\epsilon;
        \end{align*}
        \item If $\beta^2>\frac{1}{4}$, then
        \begin{align*}
            \langle t\rangle^{\frac12}\|U^{x}_{\neq}(t)\|_{L^2}+\langle t\rangle^{\frac32} \|U^y(t)\|_{L^2}+\langle t\rangle^{\frac12}\|\Theta_{\neq}(t)\|_{L^2}\lesssim \epsilon.
        \end{align*}
        \end{enumerate}
        The \(\langle t\rangle^{-1}\) decay obtained for \(U^2_{\neq}\) in our 3D result coincides with the limiting decay exponent for $U^y$ in the 2D linear estimate as \(\beta\to0^+\). For each \(\beta>0\), however, the 2D estimate gives a stronger polynomial decay power. It remains open to determine whether the 3D inviscid damping rates improve with \(\beta\), and how the nonlinear stability threshold depends on \(\beta\).
    \end{remark}         
	\noindent\textbf{Main idea of proof:} The zero-mode components $(u^2_0,u^3_0,\theta_0)$ possesses a  dispersive structure. Following the ideas of \cite{LSWZ25} for $\kappa=\mu$, we establish a similar Strichartz-type estimate \eqref{main-u230} for $\kappa\approx\mu$ (see Lemma \ref{lem-stri} for details). In contrast, the first component $u^1_0$ is governed essentially by a heat equation. To exploit this structure, we use the quantity $v^1_0=u^1_0+\beta^{-1}\theta_0$, which satisfies  
	\begin{equation*}
		\partial_t v^1_0+(u\cdot v^1)_0=\nu\Delta v^1_0,
	\end{equation*}
    where we set $\mu=\kappa$ for simplify. Since the second component of velocity $u_0^2$ losses one more derivative than $u_0^3$ in the Strichartz-type estimates (see Proposition \ref{prop-stri}), the leading nonlinear contribution comes from the interaction of zero-mode
	\begin{align*}
		\int_{0}^{T}\left\langle u^{2,hi}_0\partial_y v^{1,lo}_0, v^1_0\right\rangle_{H^s} \d t,
	\end{align*}
    while the corresponding term $\int_{0}^{T}\left\langle u^{3,hi}_0\partial_z v^{1,lo}_0, v^1_0\right\rangle_{H^s} \d t$ for $u_0^{3,hi}$ can be resolved by integration by parts (see the estimate of $T^2_3$). 
	Using the equation
    \begin{equation*}
		u^2_0=\partial_t \theta_0+(u\cdot\nabla\theta)_0-\nu \Delta \theta_0.
	\end{equation*}
	and integration by parts in time, we can transform the above leading interaction into
	\begin{align*}
		\int_{0}^{T}\left\langle \theta^{hi}_0\partial_y v^{1,lo}_0, (u_0\cdot \nabla v^1_0)\right\rangle_{H^s}dt+\cdots,
	\end{align*}
	which can be controlled within the bootstrap framework.	
	
	At the higher regularity level, the nonlinear estimates lead to the additional growth. Indeed, the leading contribution in the $H^{s+1}$ energy estimate for $v^1_0$ is of the form
	\begin{align*}
		\int_{0}^{T}\left\langle \nabla (u^{3,hi}_0\partial_z v^{1,lo}_0), \nabla v^1_0\right\rangle_{H^s}dt
		\lesssim\|u^3_0\|_{L^\infty H^{s+1}}\|\nabla v^1_0\|^2_{L^2H^s}\lesssim\nu^{-1}\epsilon^3.
	\end{align*}
	For initial datum of size $\epsilon\sim\nu^{\frac{2}{3}+}$, this indicates a growth of order $\nu^{-\frac{1}{6}}$ for the $H^{s+1}$ norm of $v^1_{0}$. For even higher regularity, the corresponding estimates would lead to a more severe loss. 
    In particular, the interaction $u^{3,lo}_{\neq}\partial_{zzy}v^{1,hi}_0$, which appearing in the equation for $\partial_{zy}v^1_{\neq}$, becomes problematic. 
	To avoid such a direct high-regularity estimate of $U^1_{\neq}$, we exploit the incompressibility  condition and introduce the good unknown
	\begin{align*}
		W_{\neq}=(pq)^{\frac{1}{2}}U^3_{\neq}-\partial_{YZ}^L{(pq)^{-\frac{1}{2}}}G_{\neq},
	\end{align*}
	see \eqref{def-cot} and \eqref{def-w} for the precise definition and corresponding equation. The energy estimate for $W_{\neq}$ incurs only a mild logarithmic loss of order $|\ln\nu|^{1+a_0}$.
	
    
    \section{Preliminary}
    \subsection{Notation}
    Throughout this paper, for $r,s \in\mathbb{R}$, we define
	\begin{equation*}
		|r,  s |:=\sqrt{r^2+s^2},\quad\langle r \rangle:=\sqrt{1+|r|^2}.
	\end{equation*}
	We use the notation $r\lesssim s$ to express $r\leq Cs$ for some constant $C>0$ independent of the parameters of interest such as $\nu,\mu$. The Fourier transform of a function $f$ is denoted by
	\begin{equation*}
		\mathcal{F}(f)(k,\eta,l)=\hat{f}_{k,\eta,l}=\hat{f}(k,\eta,l)=\int_{\mathbb{T}\times\mathbb{R}^2}e^{-i(kx+\eta y+lz)}f(x,y,z)\mathrm{d}x\mathrm{d}y\mathrm{d}z.
	\end{equation*}
	For a Fourier multiplier $A$, we define $Af=\mathcal{F}^{-1}(A(k,\eta,l)\hat{f}(k,\eta,l))$. 
	
	For a function of space and time $f(t,x,y,z)$ defined on a time interval $(a,b)$, we denote the space $L^p(a,b;H^s)$ for $1\leq p\leq\infty$ by the norm
	\begin{equation*}
		\|f\|_{L^p(a,b;H^s)}:=\|\|f\|_{H^s}\|_{L^p(a,b)}.
	\end{equation*}
	When the time interval is clear from context or mentioned explicitly elsewhere, we use the shorthand notation $\|f\|_{L^p(a,b;H^s)}=\|f\|_{L^p H^s}$.
    \subsection{Reformulation of the systems}
	To mod out the fast mixing of Couette flow, we take the following coordinate transformations:
	\begin{equation*}
		X=x-yt,\ Y=y,\ Z=z,
	\end{equation*}
    and introduce the notations
    \begin{align}
        \nabla_L:=(\partial_X,\partial_Y^L,\partial_Z)=(\partial_X,\partial_Y-t\partial_X,\partial_Z),\quad \Delta_L=\nabla_L\cdot\nabla_L.\label{def-cot}
    \end{align}
    We denote the symbol associted to $-\Delta_L,-\Delta_{X,Z}$ by
    \begin{align*}
        p=k^2+(\eta-kt)^2+l^2, \quad q=k^2+l^2,
    \end{align*}
    respectively. For a function $\phi(t,x,y,z)$, we define the corresponding quantity $\Psi(t,X,Y,Z)=\psi(t,x,y,z)$ in the new frame. In particularly, we set $(U,\Theta)(t,X,Y,Z)=(u,\theta)(t,x,y,z)$. The system \eqref{eq-main} is then rewritten as follows
	\begin{equation*}
		\left\{
		\begin{aligned}
			&\partial_t U+(U^2,0,0)^T+(U\cdot\nabla_L)U+\nabla_L P^L+\nabla_LP^{NL}=\kappa \Delta_L U-\beta(0,\Theta,0)^T,\\
			&\partial_t \Theta-\beta U^2+(U\cdot\nabla_L)\Theta=\mu\Delta_L\Theta,\\
            &\Delta_LP^L=-2\partial_XU^2-\beta\partial_Y^L\Theta,\quad \Delta_LP^{NL}=-\partial_i^LU^j\partial_j^LU^i.
		\end{aligned}
		\right.
	\end{equation*}
	
	For nonzero modes, we consider the coupled system
	\begin{align}
		(G_{\neq},\ F_{\neq})=(pU^2_{\neq},(pq)^{\frac{1}{2}}\Theta_{\neq}),\ 
	\end{align}
	which satisfies
	\begin{equation}
		\left\{\begin{aligned}
			&\partial_t G+\beta\left(\frac{q }{p}\right)^{\frac{1}{2}}F-\kappa \Delta_L G=-p(U\cdot\nabla_LU^2)-\partial_Y^L(\partial_i^LU^j\partial_j^LU^i),\\
			&\partial_t F-\frac{1}{2}\frac{\partial_t p}{p}F-\beta\left(\frac{q }{p}\right)^{\frac{1}{2}}G -\mu \Delta_L F=-(pq)^{\frac{1}{2}}(U\cdot\nabla_L\Theta).
		\end{aligned}\right.
	\end{equation}
    To weaken the linear pressure effect, we introduce the good unknown
    \begin{align}
        W_{\neq}=\left((pq)^{\frac{
				1}{2}}U^3_{\neq}-\frac{\partial_{YZ}^L}{(pq)^{\frac{1}{2}}}G_{\neq}\right).\label{def-w}
    \end{align}
    Then the equation of $W_{\neq}$ reads
    \begin{align}
        &\partial_t W-\frac{1}{2}\frac{\partial_t p}{p}W+\frac{2\partial_{XZ}q^{\frac{1}{2}}}{p^{\frac{3}{2}}}G -\frac{\partial_{XZ}}{(pq)^{\frac{1}{2}}}G-\frac{(\partial_tp)\partial_{YZ}^L}{p^{\frac{3}{2}}q^{\frac{1}{2}}}G -\kappa\Delta_LW=\mathcal{NL}_{W},
    \end{align}
	where we denote
	\begin{equation}\label{eq-w}
		\begin{aligned}
			\mathcal{NL}_{W}=&-(pq)^{\frac{1}{2}}(U\cdot\nabla_LU^3)+\frac{\partial_{YZ}^L}{(pq)^{\frac{1}{2}}}p(U\cdot\nabla_LU^2)\\
			&-\partial_Z(\frac{q}{p})^{\frac{1}{2}}(\partial_i^LU^j\partial_j^LU^i)+\partial_Y^L\frac{\partial_{YZ}^L}{(pq)^{\frac{1}{2}}}(\partial_i^LU^j\partial_j^LU^i).
		\end{aligned}
	\end{equation}
	Moreover, by using the definition of $(G,F,W)$ and divergence-free condition, the nonzero modes of $(U_{\neq},\Theta_{\neq})$ can be explicitly expressed as:
	\begin{align*}
		&\partial_XU^1_{\neq}=-\partial_Y^LU^2_{\neq}-\partial_ZU^3_{\neq},\\
		&U^2_{\neq}=p^{-1}G_{\neq},\\
		&U^3_{\neq}=(pq)^{-\frac{1}{2}}\left(W_{\neq}+\frac{\partial_{YZ}^L}{(pq)^{\frac{1}{2}}}G_{\neq}\right),\\
		&\Theta_{\neq}=(pq)^{-\frac{1}{2}}F_{\neq}.
 	\end{align*}
	For zero-mode, to capture the dispersive property of $(U^2_0,U^3_0,\Theta_0)$, we introduce the quantity
	\begin{align*}
		\quad p_0^{-\frac{1}{2}}\Omega^1_0=p_0^{-\frac{1}{2}}(\partial_YU^3_0-\partial_ZU^2_0).
	\end{align*}
    Here $p_0=-\Delta_{y,z}$.	Then the system of $(\Theta_0,p_0^{-\frac{1}{2}}\Omega^1_0)$ satisfies
	\begin{equation}\label{eq-the0}
	    \left\{
        \begin{aligned}
            &\partial_t \Theta_0-\beta\partial_Zp_0^{-\frac{1}{2}}(p_0^{-\frac{1}{2}}\Omega^1_0)-\mu \Delta\Theta_0=-(U\cdot\nabla \Theta)_0,\\
		&\partial_t (p_0^{-\frac{1}{2}}\Omega^1_0)-\beta\partial_Zp_0^{-\frac{1}{2}}\Theta_0-\kappa \Delta(p_0^{-\frac{1}{2}}\Omega^1_0)=-p_0^{-\frac{1}{2}}[\partial_Y(U\cdot\nabla_LU^3)-\partial_Z(U\cdot\nabla_LU^2)]_0.
        \end{aligned}\right.
	\end{equation}
    To suppress the lift-up effect, we use the unknown
    \begin{align*}
        V^1_0=U^1_0+\beta^{-1}\Theta_0,
    \end{align*}
    which solve the equation
    \begin{align*}
		&\partial_t V^1_0-\kappa \Delta V^1_0=\beta^{-1}(\mu-\kappa)\Delta\Theta_0-(U\cdot\nabla_LV^1)_0.
	\end{align*}
	Then the zero-mode of velocity can be recovered by
	\begin{align}\label{eq-zero}
		&U^1_0=V^1_0-\beta^{-1}\Theta_0,\quad
		U^2_0=\partial_Zp_0^{-\frac{1}{2}}(p_0^{-\frac{1}{2}}\Omega^1_0),\quad
		U^3_0=-\partial_Yp_0^{-\frac{1}{2}}(p_0^{-\frac{1}{2}}\Omega^1_0).
	\end{align}

    \subsection{The Fourier multipliers}
	We introduce the following multipliers to capture the stability mechanisms of the above systems. The first two multipliers $M^1$ and $M^2$ are now standard for studying the stability of Couette flow. They are determined by the equations
	\begin{equation*}
		\left\{\begin{array}{l}
			\displaystyle\frac{\partial_tM_1}{M_1}=-\frac{k^2+|kl|}{k^2+(\eta-kt)^2+l^2},\ \mathrm{for}\ k\neq0,\\
			\displaystyle\frac{\partial_tM_2}{M_2}=-\frac{\nu^{1/3}k^2}{k^2+\nu^{2/3}(\eta-kt)^2},\ \mathrm{for}\ k\neq0,\\
			M_j(t,0,\eta,l)=M_j(0,k,\eta,l)=1,\ j=1,2.
		\end{array}\right.
	\end{equation*}
	The multiplier $M_1$ helps us to obtain inviscid damping and control terms arising from the linear pressure. The multiplier $M_2$ is used to capture the enhanced dissipation, which satisfies the key property:
	\begin{equation}\label{prom3}
		\nu^{1/3}\lesssim\nu(\eta-kt)^2-\frac{\partial_tM_2}{M_2}.
	\end{equation}
	The proofs of these properties are standard and can be found in \cite{BGM17}. 
	We also need a multiplier $M_3$ introduced in \cite{LSWZ25}, which helps us to control the nonlinear cascades and linear pressure, defined as
    \begin{equation*}
		\left\{
		\begin{aligned}
			&\frac{\partial_tM_3}{M_3}=-\sum_{n\in\mathbb{Z}\backslash\{0\}}\langle t-\frac{\eta}{k}\rangle^{-1}[\ln(e+\langle t-\frac{\eta}{n}\rangle )]^{-1-a_0}\langle k-n\rangle^{-1-a_0},\\
			&M_3(0,k,\eta,l)=1.
		\end{aligned}
		\right.
	\end{equation*}
    \color{blue}
    \color{black}
	for any $a_0>0$. The multiplier $M_3$ satisfies the following useful property (see \cite{LSWZ25}):
	\begin{equation}\label{prop-m3}
	     \begin{aligned}
		&M_3(t,k,\eta)\approx1,\\
        & \langle t-\frac{\eta}{k}\rangle^{-1}|\ln(1+\langle t-\eta/k\rangle)|^{-1-a_0} \lesssim \frac{ \partial_t M_3}{M_3} (t,k,\eta),\\
        & \langle t-\frac{\eta}{k}\rangle^{-1}|\ln(1+\langle t-\eta/k\rangle)|^{-1-a_0} \lesssim \frac{ \partial_t M_3}{M_3} (t,k_1,\eta_1) \langle k-k_1, \eta-\eta_1 \rangle ^{2+a_0(2+a_0)},
	\end{aligned}
	\end{equation}
	for $k\neq 0$.
    To control the linear stretching term, we introduce a multiplier $m$, defined as
	\begin{equation*}
		\frac{\partial_t m}{m}=\left\{\begin{aligned}
			&\frac{2k(\eta-kt)}{k^2+(\eta-kt)^2+l^2},\quad if\ 0\leq t-\eta/k\leq4\nu^{-1/3},\\
			&0,\quad otherwise,
		\end{aligned}\right.
	\end{equation*}
	which satisfies the following properties (see \cite{BGM17}):
	\begin{subequations}\label{propm}
			\begin{align}
			&\max\left\{\langle t\rangle^{-2},\nu^{2/3}\right\}\lesssim m\leq 1;\label{prop1m}\\
			& k^2+l^2\lesssim p m;\label{prop2m}\\
			& m(k,\eta,l)\lesssim \langle\eta-\eta_1,l-l_1\rangle^2m(k,\eta_1,l_1),\label{prop3m} \\
            & \left|\frac{\partial_t p}{2p}\right| \leqslant -\frac{\partial_t m}{2m} + \frac{\nu}{16} p. 
		\end{align}
    \end{subequations}
    Finally, we define the combined multiplier $M$ as
	\begin{align*}
		&M(t,k,\eta,l)=
			e^{a_1\mathbf{1}_{k\neq0}\nu^{1/3}t}(M_1M_2M_3)(t,k,\eta,l),
	\end{align*}
	where $a_1=\frac{1}{100}$ is a sufficiently small constant.
    \subsection{Bootstrap argument}
    To prove Theorem \ref{mainthe}, we employ a bootstrap argument. Henceforth, all time norms are taken over the interval $[0,T]$ unless otherwise specified.
    We define
\begin{align*}
        \|f\|_{A^s} &:=\|f\|_{L^{\infty}H^s}+\nu^{\frac{1}{2}}\|\nabla_L f\|_{L^2H^s}
                 +\sum_{i=1,2,3}\left\|\sqrt{-\partial_tM_i/M_i}f\right\|_{L^2H^s}.
\end{align*}
    \textbf{Bootstrap hypotheses:} Fix a large constant $C_0$ (to be determined later) and let $T>0$ be the maximal time such that the following estimates hold on $[0,T]$:
    \begin{subequations}\label{bh}
        \begin{align}
    	&\|M\wtb V^1_0\|_{A^s}\leq10C_0\epsilon,\label{bs-v1}\\
    	&\|M\nabla V^1_0\|_{A^{s}}\leq10C_0\nu^{-1/6}\epsilon,\label{bs-v1h}\\
    	&\|M\wtb(U^2_0,U^3_0,\Theta_0)\|_{A^{s+1}}\leq10\epsilon,\label{bs-u20}\\
    	&\|m^{\frac{1}{2}}M(G_{\neq},F_{\neq})\|_{A^s}\leq 10\epsilon,\label{bs-gf}\\
    	&\|m^{\frac{1}{2}}MW_{\neq}\|_{A^s}\leq 10C_0 |\ln{\nu} |^{1+a_0}\epsilon.\label{bs-w}
    \end{align}
    \end{subequations}

    The next proposition shows that $T$ is in fact infinite.

\begin{proposition}\label{bsup1}
    Under the assumptions of Theorem~\ref{mainthe} and the bootstrap hypotheses \eqref{bh}, there exist constants $\epsilon_0>0$ such that if $\epsilon<\epsilon_0\nu^{\frac{2}{3}}|\ln\nu|^{-2-2a_0}$, then estimates \eqref{bs-v1}--\eqref{bs-w} hold on $[0,T]$ with all constants on the right‑hand side replaced by one half of their original values. By continuity this implies $T=+\infty$.
\end{proposition}

Finally, we derive several useful estimates that follow directly from the bootstrap hypotheses.
    \begin{lemma}\label{est-fun}
        Under the assumptions of Theorem \ref{mainthe}, it holds that
        \begin{align*}
        &\|q^{\frac{1}{2}}\partial_X U^1_{\neq}\|_{L^\infty H^s}+\nu^{\frac{1}{6}}\|q^{\frac{1}{2}}\partial_X U^1_{\neq}\|_{L^2 H^s}\lesssim\epsilon|\ln{\nu} |^{1+a_0},\\
		&\|(pq)^{\frac{1}{2}}U^2_{\neq}\|_{L^\infty H^s}+\|\partial_Xq^{\frac{1}{2}}U^2_{\neq}\|_{L^2 H^s}+\nu^{\frac{1}{6}}\|(pq)^{\frac{1}{2}}U^2_{\neq}\|_{L^2 H^s}\lesssim\epsilon,\\
		&\|q U^3_{\neq}\|_{L^\infty H^s}+\nu^{\frac{1}{6}}\|q U^3_{\neq}\|_{L^2 H^s}\lesssim\epsilon|\ln{\nu} |^{1+a_0},,\\
		&\|q\Theta_{\neq}\|_{L^\infty H^s}+\nu^{\frac{1}{6}}\|q\Theta_{\neq}\|_{L^2H^s}\lesssim\epsilon,\\
		&\|\wtb U^1_{0}\|_{L^\infty H^s}+\nu^{\frac{1}{6}}\|U^1_{0}\| _{L^\infty H^{s+1}}+\|\wtb (U^2_{0},U_0^3,\Theta_0) \|_{L^\infty H^{s+1}} \lesssim\epsilon.
	\end{align*}
    \end{lemma}
    \begin{proof}
        Lemma \eqref{est-fun} follows from the bootstrap hypotheses, \eqref{prop2m} and the divergence free condition, we thus omit its proof. 
    \end{proof}
    It is worth noting that we do not obtain the estimates of $\|m^{\frac{1}{2}}\partial_Zp^{\frac{1}{2}}U^1_{\neq}\|_{H^s}$, since the use of divergence-free condition would loss one derivative.
	
	 
    
	\section{Estimates on zero mode}
	\subsection{Dispersive and Strichartz estimates on $(\Theta_0, p_0^{-\frac{1}{2}} \Omega^1_0)$}
	In this subsection, we establish the Strichartz estimates for zero-mode $(U^2_0,U^3_0,\Theta_0)$, which is stated in the following proposition.
    \begin{proposition}\label{prop-stri}
		Under the assumption of Theorem \eqref{mainthe} and bootstrap hypotheses \eqref{bh}, it holds that for $0<\delta<\frac{1}{12}$:
		\begin{align}
			~&\left\| \wtb (\Theta_0,U_0^{2,3}) \right \|_{L^2 W^{s-2,\infty}} \lesssim \nu^{-1/12-\delta}\epsilon, \label{est-2l}\\
			~& \left\| \wtb U_0^{3} \right \|_{L^2 W^{s-1,\infty}} + \nu^{1/2} \nm{\wtb \partial_Y (\Theta_0, U_0^2,U_0^3)}{L^1 W^{s-1,\infty}} \lesssim  \nu^{-\delta} \epsilon \label{est-3}.
		\end{align}
	\end{proposition}
  	Let $\mathcal{R}_3=\partial_z(-\Delta_{y,z})^{-\frac{1}{2}}$. Consider the following linear system on $\mathbb{R}^2$:
  	\begin{equation}\label{eq-linear}
  		\left\{\begin{aligned}
  			&\partial_t \phi_1-\beta \mathcal{R}_3 \phi_2-\kappa\Delta \phi_1=0,\\
  			&\partial_t \phi_2-\beta \mathcal{R}_3 \phi_1-\mu\Delta \phi_2=0,\\
  			&(\phi_1,\phi_2)|_{t=0}=(\phi_{1,in},\phi_{2,in}).
  		\end{aligned}\\
  		         \right.
  	\end{equation}
  	Here $\beta>0$, $\kappa,\mu>0$, and $\kappa\approx\mu$. Let $S(t)$ be the solution operator of \eqref{eq-linear} and set
    \begin{align*}
        \xi=(\eta,l),\quad \sigma=\frac{l}{|\xi|},\quad \nu_1=\frac{\kappa+\mu}{2},\quad \nu_2=\frac{\mu-\kappa}{2}.
    \end{align*}
    Taking the Fourier transform of 
  	\eqref{eq-linear}, we obtain
  	\begin{equation*}
  		\partial_t \hat\Phi=M(\xi)\hat\Phi,
  	\end{equation*}
  	where we denote
  	\begin{equation*}
  		\Phi=(\phi_1,\phi_2)^{T},\quad M(\xi)=-\nu_1|\xi|^2I+N(\xi),\quad N(\xi)=\left(\begin{array}{cc}
  		     \nu_2|\xi|^2&i\beta\sigma\\
  		     i\beta\sigma&-\nu_2|\xi|^2
  		\end{array}\right).
  	\end{equation*}
  	Since $-\nu_1|\xi|^2I$ commutes with $N$, it holds that
  	\begin{align*}
  		e^{tM}=e^{-\nu_1|\xi|^2 t}e^{tN}.
  	\end{align*} 
  	Noting the fact that
  	\begin{equation*}
  		N(\xi)^2=(\nu_2^2|\xi|^4-\beta^2\sigma^2)I,
  	\end{equation*}
    then we have 
    \begin{equation}\label{eq-s}
        \hat{S}(t,\xi)=\left\{\begin{aligned}
            &e^{-\nu_1|\xi|^2 t}\left(\cos(t\psi)I+\frac{sin(t\psi)}{\psi}N\right),\quad \mathrm{for}\ |\beta\sigma|>|\nu_2||\xi|^2,\\
            &e^{-\nu_1|\xi|^2 t}\left(\cosh(t\psi)I+\frac{sinh(t\psi)}{\psi}N\right),\quad \mathrm{for}\ |\beta\sigma|\leq|\nu_2||\xi|^2,
        \end{aligned}\right.
    \end{equation}
    with
  	\begin{equation*}
  		\psi=|\beta^2\sigma^2-\nu_2^2|\xi|^4|^{\frac{1}{2}}.
  	\end{equation*}
    Recalling the equations \eqref{eq-the0}, the nonlinear solution can be expressed as 
    \begin{align}
        \varPhi(t)=S(t)\varPhi_{in}+\int_{0}^{t}S(t-\tau)\mathcal{NL}(\tau)d\tau,\label{eq-phi}
    \end{align}
    where we denote
    \begin{align*}
        &\varPhi=(\Theta,p_0^{-\frac{1}{2}}\Omega_0^1)^{\top},\\
        &\mathcal{NL}=-\left((U\cdot\nabla \Theta)_0,\ p_0^{-\frac{1}{2}}[\partial_Y(U\cdot\nabla_LU^3)-\partial_Z(U\cdot\nabla_LU^2)]_0\right)^{\top}.
    \end{align*}
    We summarize the properties of operator $S(t)$ in the following lemma. All vector-valued norms below are taken componentwise.
    \begin{lemma}\label{lem-stri}
        It holds that for any $F\in\mathcal{S}(\mathbb{R}^2)$ and any small constant $\delta_1>0$:
        \begin{align}
  			&\|S(t)F\|_{L^\infty}\lesssim(1+t)^{-\frac{1}{2}}(\nu t)^{-\delta_1}\|F\|_{W^{2,1}},\label{est-d1}\\
  			&\|S(t)F\|_{L^2L^\infty}\lesssim\nu^{-\frac{1}{4}}\|F\|_{H^{1/2+\epsilon}},\label{est-d2}\\
  			&\|S(t)\partial_y|\nabla_{y,z}|^{b-1}F\|_{L^2L^\infty}\lesssim\nu^{-(\delta_1+\frac{b}{2})}\|F\|_{H^1},\label{est-d3}
  		\end{align}
    where $b=0,1$.
    \end{lemma}
    The proof of this lemma is provided in the appendix.
	\begin{proof}[Proof of proposition \ref{prop-stri}]
	Applying Minkowski's inequality, Fubini's thoerem, and lemma \ref{lem-stri}, we have
    \begin{align*}
        &\left\| S(t) F\right \| _{L^2L^{\infty}} \lesssim \nu^{-\delta} \left \| F\right \| _{W^{2,1}},\\
		&\left\| S(t) \mathcal{R}_2 F \right \| _{L^2L^{\infty}} \lesssim \nu^{-\delta } \left \| F\right \| _{H^1},\\
		&\left\| S(t) \partial_Y F \right \| _{L^1 L^{\infty}} \lesssim \nu^{-1/2-\delta } \left \| F\right \| _{H^1},\\
		&\left\| \int_0^t S(t-s) H(s) \d s \right \| _{L^2L^{\infty}} \lesssim \nu^{-1/4} \left \| H\right \| _{L^1H^{\frac{1}{2}+}},\\
		&\left\| \int_0^t S(t-s) \mathcal{R}_2 H(s) \d s \right \| _{L^2L^{\infty}} \lesssim \nu^{-\delta} \left \| H\right \| _{L^1 H^1},\\
		&\left\| \int_0^t S(t-s) \partial_Y H(s) \d s \right \| _{L^1 L^{\infty}} \lesssim \nu^{-1/2-\delta} \left \| H\right \| _{L^1 H^1},
	\end{align*}
	In view of \eqref{eq-phi} and \eqref{eq-zero}. the above six estimates imply that
	\begin{equation}\label{ineq-dl}
		\begin{aligned}
			~&\left\|\wtb (\Theta_0,U_0^{2},U_0^{3}) \right \|_{L^2 W^{s-2,\infty}}\\
			\lesssim & \nu ^{-\delta} \left \| \wtb  (U_{0,in},\Theta_{0,in})  \right \| _{W^{s,1}}+\nu^{-1/4} \left \|  \wtb \left ( U_0\cdot \nabla (\Theta_0, U_0^2,U_0^3) \right ) \right \| _{L^1H^{s-1}} \\
			~& +\nu ^{-1/4} \left \|  \wtb  \left(U_{\neq}\cdot \nabla_L (\Theta_{\neq}, U_{\neq}^2, U_{\neq }^3)\right)_0 \right \|_{L^1H^{s-1}},
		\end{aligned}
	\end{equation}
	and
	\begin{equation}\label{ineq-gc}
	    \begin{aligned}
		~& \left\|\wtb  U_0^{3} \right \|_{L^2 W^{s-1,\infty}}+\nu^{1/2}\left\| \wtb \partial_Y (\Theta_0, U_0^2, U_0^{3}) \right \|_{L^1 W^{s-1,\infty}}\\
		\lesssim & \nu ^{-\delta} \left \|  \wtb  (U_{0,in},\Theta_{0,in}) \right \| _{H^s}+\nu^{-\delta} \left \| \wtb  \left ( U_0\cdot \nabla (\Theta_0, U_0^2,U_0^3) \right ) \right \| _{L^1H^s} \\
		~& +\nu ^{-\delta} \left \| \wtb  \left(U_{\neq}\cdot \nabla_L (\Theta_{\neq}, U_{\neq}^2, U_{\neq }^3)\right)_0 \right \|_{L^1H^s},
	\end{aligned}
	\end{equation}
	where we choose $\frac{1}{2}+=1$.

	First, we treat the good components estimates for $U^3_0$ and $\partial_Y(\Theta_0,U^2_0,U^3_0)$. By using divergence-free condition, Moser's type inequality, and bootstrap assumptions, we obtain that
	\begin{align*}
		~&\left \| \wtb ( U_0\cdot \nabla (\Theta_0, U_0^2,U_0^3)) \right \| _{L^1H^{s}}=\left \|\wtb ( \nabla \cdot \left( (U_0^2,U_0^3) \otimes (\Theta_0, U_0^2,U_0^3)\right) \right \| _{L^1H^{s}}\\
		\lesssim  & \nm{\wtb(\Theta_0, U_0^2,U_0^3)}{L^2 L^{\infty}} \nm{\nabla \wtb (\Theta_0, U_0^2,U_0^3)}{L^2 H^s} \\
		\lesssim & \epsilon \nu^{-1/2} \nm{ \wtb (\Theta_0, U_0^2,U_0^3)}{L^{2} L^{\infty}},
	\end{align*}
	and 
	\begin{align*}
		&\|\wtb (U_{\neq}\cdot\nabla_L(\Theta_{\neq},U^2_{\neq},U^3_{\neq}))\|_{L^1H^s}\\
		\lesssim&\|\wtb U^{1,3}_{\neq}\|_{L^2H^s}\| \wtb q (\Theta_{\neq},U^2_{\neq},U^3_{\neq})\|_{L^2H^s}+\|\wtb U^2_{\neq}\|_{L^2H^s}\|\nabla_L\wtb(\Theta_{\neq},U^2_{\neq},U^3_{\neq})\|_{L^2H^s}\\
        \lesssim & \nm{m^{\frac12} p^{\frac12} U_{\neq}^{1,3}}{L^2 H^{s}} \nm{ m^{\frac12} (G_{\neq}, F_{\neq}, W_{\neq })}{L^2 H^{s}}+ \nm{m^{\frac12} \partial_X p^{-\frac12} G_{\neq}}{L^2 H^s} \nm{\nabla_L m^{\frac12}  (G_{\neq}, F_{\neq}, W_{\neq })}{L^2 H^{s}} \\
		\lesssim& \nu^{-1/6-1/6}\epsilon^2 |\ln{\nu}|^{2+2a_0} +\nu^{-1/2}\epsilon^2 |\ln{\nu}|^{1+a_0} .
	\end{align*}
	Then it holds that
	\begin{align*}
		~& \left\| \wtb U_0^{3} \right \|_{L^2 W^{s-1,\infty}} + \nu^{1/2} \nm{\wtb \partial_Y (\Theta_0, U_0^2,U_0^3)}{L^1 W^{s-1,\infty}}  \\
		\lesssim &  \nu ^{-\delta} \left \|  ( \Theta_{0,in},U_{0,in}^{2,3})  \right \| _{H^s}+ \epsilon \nu^{-1/2-\delta} \nm{ (\Theta_0, U_0^2,U_0^3)}{L^{2} L^{\infty}} +\epsilon^2 \nu^{-1/2-\delta}  | \ln{\nu}|^{1+a_0}
	\end{align*}
	Turning to the estimates of low regularity, the above estimate implies that
	\begin{align*}
		~&\left \|\wtb ( U_0\cdot \nabla (\Theta_0, U_0^2,U_0^3)) \right \| _{L^1H^{s-1}}\\
		\lesssim  &\left \|\wtb( U_0^2 \partial_Y (\Theta_0, U_0^2,U_0^3)) \right \| _{L^1 H^{s-1}} +\left \|\wtb (U_0^3 \partial_Z (\Theta_0, U_0^2,U_0^3)) \right \| _{L^1 H^{s-1}}  \\
		\lesssim & \nm{\wtb U_0^2}{L^{\infty} H^{s-1}} \nm{\wtb \partial_Y (\Theta_0, U_0^2,U_0^3)}{L^1 W^{s-1,\infty}}\\
		& + \nm{\wtb U_0^3}{L^2 W^{s-1,\infty}} \nm{\partial_Z\wtb (\Theta_0, U_0^2,U_0^3)}{L^{2} H^{s-1}  } \\    
		\lesssim & \epsilon \nu^{-1/2} \left( \nu^{1/2} \nm{\wtb \partial_Y (\Theta_0, U_0^2,U_0^3)}{L^1 W^{s-1,\infty}} + \nm{\wtb U_0^3}{L^2 W^{s-1,\infty}}  \right) \\ 
		~\lesssim & \epsilon \nu^{-1/2} \left( \nu ^{-\delta_3} \left \| \wtb (\Theta_0,  U_{0,in}^{2,3})  \right \| _{H^s}+ \epsilon \nu^{-1/2-\delta} \nm{ \wtb (\Theta_0, U_0^{2,3})}{L^{2} L^{\infty}} +\epsilon^2 \nu^{-1/2-\delta} | \ln{\nu} |^{1+a_0} \right).  
	\end{align*}
	
    For the other nonlinear term, it follows from divergence-free condition that
    \begin{align*}
        ~ & \nm{\wtb \left( U_{\neq}\cdot  \nabla_L (\Theta_{\neq}, U_{\neq}^2, U_{\neq }^3) \right)_0}{L^1 H^{s-1}}\\
        = & \nm{\wtb \left( U_{\neq}\otimes (\Theta_{\neq}, U_{\neq}^2, U_{\neq }^3) \right)_0}{L^1 H^{s}}\\
        \lesssim &  \nm{ m^{1/2} p^{1/2} U_{\neq}}{L^2 H^{s}} \nm{m^{1/2} p^{1/2} (\Theta_{\neq}, U_{\neq}^2, U_{\neq }^3)}{L^2 H^{s}}  \\  
        \lesssim & \epsilon^2 \nu^{-1/6-1/6}|\ln{\nu}|^{2+2 a_0}.
    \end{align*}
	Substituting these two estimates into \eqref{ineq-dl}, we conclude that
	\begin{align*}
		~&\left\|  \wtb (\Theta_0,U_0^2,U_0^3) \right \|_{L^2 W^{s-2,\infty}}\\
		\lesssim & \nu ^{-\delta} \left \| (\Theta_0,  U_{0,in}^{2,3})  \right \| _{W^{s,1}}+\nu^{-3/4} \epsilon \left( \nu ^{-\delta_3} \epsilon+\epsilon\nu^{-1/2-\delta}\|\wtb (\Theta_0,U^{2,3}_0)\|_{L^2L^{\infty}}+\epsilon^2 \nu^{-1/2-\delta} |\ln{\nu}|^{1+a_0}  \right)  \\
        &+\nu^{-7/12} |\ln{\nu}|^{2+2a_0} \epsilon^2\\
		\lesssim & \epsilon \nu ^{-\delta} + \epsilon \nu^{-\frac23} \epsilon \nu^{-\frac{1}{12}-\delta_3} + (\epsilon \nu^{-\frac{5}{8}-\delta/2})^2 \|\wtb (\Theta_0,U^{2,3}_0)\|_{L^2L^{\infty}} + \epsilon^3 \nu^{-\frac{5}{4}-\delta}|\ln{\nu}|^{1+a_0} + \nu^{-\frac{7}{12}} |\ln{\nu}|^{2+2a_0} \epsilon^2
	\end{align*}
    Moreover, we deduce that
    \begin{align*}
		~& \left\| \wtb U_0^{3} \right \|_{L^2 W^{s-1,\infty}} + \nu^{1/2} \nm{\wtb \partial_Y (\Theta_0, U_0^2,U_0^3)}{L^1 W^{s-1,\infty}}  \\
		\lesssim &  \nu ^{-\delta_3} \left \|  ( \Theta_{0,in},U_{0,in}^{2,3})  \right \| _{H^s}+ \epsilon \nu^{-1/2-\delta} \nm{ (\Theta_0, U_0^2,U_0^3)}{L^{2} L^{\infty}} +\epsilon^2 \nu^{-1/2-\delta}  | \ln{\nu}|^{1+a_0}.
	\end{align*}
	Choosing $0<\delta < \frac{1}{12}$ sufficiently small, it follows from $\epsilon \ll \nu^{2/3}$ that 
	\begin{align*}
		~&\left\| \wtb (\Theta_0,U_0^{2,3}) \right \|_{L^2 W^{s-2,\infty}} \lesssim \nu^{-1/12 -\delta}  \epsilon, \\
		~& \left\| \wtb U_0^{3} \right \|_{L^2 W^{s-1,\infty}} + \nu^{1/2} \nm{\wtb \partial_Y (\Theta_0, U_0^2,U_0^3)}{L^1 W^{s-1,\infty}}\\
        \lesssim &  \nu ^{-\delta} \left \|  ( \Theta_{0,in},U_{0,in}^{2,3})  \right \| _{H^s}+ \epsilon^2 \nu^{-1/2-\delta-1/12} +\epsilon^2 \nu^{-1/2-\delta}  | \ln{\nu}|^{1+a_0} \\
        \lesssim & \nu^{-\delta} \epsilon.
	\end{align*}
    The proof of this proposition is complete.
	\end{proof}

	\subsection{Energy estimates on $(U^2_0,U^3_0,\Theta_0)$ system}
	In this subsection, we improve \eqref{bs-u20}. By incompressible condition, The standard energy estimate gives that
	\begin{align*}
		&\|\wtb M(U^2_0,U^3_0,\Theta_0)(T)\|^2_{H^{s+1}}+\nu \|\nabla\wtb M(U^2_0,U^3_0,\Theta_0)\|^2_{L^2H^{s+1}}\\
        &+\left\|\sqrt{-\frac{\partial_tM_3}{M_3}}\wtb M(U^2_0,U^3_0,\Theta_0)\right\|^2_{L^2H^{s+1}}\\
		\lesssim&\|\wtb M(U^{2,3}_{0,in},\Theta_{0,in})\|^2_{H^{s+1}}+\left|\sum_{i=2,3}\int_{0}^{T}\langle\wtb M(U\cdot\nabla_LU^i)_0,\wtb MU^i_0\rangle_{H^{s+1}}dt\right|\\
		&+\left|\int_{0}^{T}\langle\wtb M(U\cdot\nabla_L\Theta)_0,\wtb M\Theta_0\rangle_{H^{s+1}}dt\right|.
	\end{align*}
	For the interaction of zero mode, with the help of \eqref{est-2l}, product estimate, and bootstrap assumptions, one obtain that
	\begin{align*}
		&\int_{0}^{T}\langle\wtb M(U_0\cdot\nabla\Theta_0),\wtb M\Theta_0\rangle_{H^{s+1}}dt\\
		\lesssim&\| \nabla \wtb (U^2_0,U^3_0,\Theta_0)\|_{L^2L^{\infty}}\|\nabla\wtb (U^2_0,U^3_0,\Theta_0)\|_{L^2H^{s+1}}\|\wtb(U^2_0,U^3_0,\Theta_0)\|_{L^{\infty}H^{s+1}}\\
		\lesssim&\nu^{-7/12 -\delta}\epsilon^3.
	\end{align*}
	Turning to the interaction of nonzero modes, we decompose it into
	\begin{align*}
		&\int_{0}^{T}\langle\wtb M(U_{\neq}\cdot\nabla_L\Theta_{\neq}),\wtb M\Theta_0\rangle_{H^{s+1}}dt\\
		=&\int_{0}^{T}\langle\wtb M(U_{\neq}\cdot\nabla_L\Theta_{\neq}),\wtb M\Theta_0\rangle_{H^{s}}dt+\sum_{j=1,3}\int_{0}^{T}\langle\nabla \wtb M(U^j_{\neq}\partial_j\Theta_{\neq}),\wtb\nabla M\Theta_0\rangle_{\dot{H^{s}}}dt\\
		&+\int_{0}^{T}\langle\wtb \nabla M(U^2_{\neq}\partial_Y^L\Theta_{\neq}),\wtb\nabla M\Theta_0\rangle_{\dot{H^{s}}}dt\\
		\triangleq &T^1_1+T^1_2+T^1_3.
	\end{align*}
	Using the bootstrap assumption and integration by parts, we have
	\begin{align*}
		T^1_1+T^1_2\lesssim&\|\wtb U_{\neq}\|_{L^2H^s}\|\wtb \nabla_L \Theta_{\neq}\|_{L^2H^s}\|\wtb \Theta_0\|_{L^{\infty}H^s}\\
		&+\|\wtb U^{1,3}_{\neq}\|_{L^2H^s}\| \wtb q^{1/2}\Theta_{\neq}\|_{L^{\infty}H^s}\|\nabla\wtb\Theta_0\|_{L^2H^{s+1}}\\
		\lesssim&\nu^{-\frac{1}{6}}\epsilon\nu^{-\frac{1}{2}}\epsilon^2+\nu^{-\frac{1}{6}}\epsilon^2\nu^{-\frac{1}{2}}\epsilon\lesssim\nu^{-\frac{2}{3}}\epsilon^3.
	\end{align*}
	Noting the fact that
	\begin{align}\label{ineq-m31}
			\left\langle t-\frac{\eta}{k}\right\rangle^{-1}\lesssim|\ln(\nu)|^{1+a_0}\left\langle t-\frac{\eta}{k}\right\rangle^{-1}|\ln(1+\langle t-\frac{\eta}{k}\rangle)|^{-1-a_0}+\nu, 
	\end{align}
	and	
	\begin{align}\label{ineq-m32}
		U^2_{\neq}(t,k,\eta,l)\lesssim  \langle t\rangle ^{-1}	\mathcal{F}(p^{\frac{1}{2}}\langle\nabla\rangle U^2_{\neq})(t,k,\eta,l),
	\end{align}
	then it follows from lemma \ref{est-fun} and \eqref{prop-m3} that
	\begin{align*}
		T^1_3\lesssim&\|(pq)^{\frac{1}{2}}U^2_{\neq}\|_{L^2H^s}\|(pq)^{\frac{1}{2}}\Theta_{\neq}\|_{L^2H^s}\|\wtb \Theta_0\|_{L^\infty H^{s+1}}\\
		&+\sum_{k\neq0}\int_{0}^{T}\int_{\mathbb{R}^4}\langle l\rangle|\langle t\rangle^{-1}\mathcal{F}(p^{\frac{1}{2}}\langle\nabla\rangle U^2_{\neq})(t,-k,\eta-\eta_1,l-l_1)|\\
		&\quad\quad\times|(\eta_1-kt)\mathcal{F}(\langle \nabla\rangle^sp^{\frac{1}{2}}\Theta_{\neq})(t,k,\eta_1,l_1)||\mathcal{F}(\wtb\langle \nabla\rangle^{s+1}\Theta_0)(t,0,\eta,l)|d\eta d\eta' dldl'dt\\
		&+|\ln(\nu)|^{1+a_0}\sum_{k\neq0}\int_{0}^{T}\int_{\mathbb{R}^4}\langle l\rangle\left|\sqrt{-\frac{\partial_tM_3}{M_3}}\mathcal{F}(p^{\frac{1}{2}}\langle\nabla\rangle^s U^2_{\neq})(t,-k,\eta-\eta_1,l-l_1)\right|\\
		&\quad\quad\times  \left|\langle t\rangle|\eta_1,k|^{2+\frac{a_0(2+a_0)}{2}}\mathcal{F}(p^{1/2}\Theta)_{\neq}(t,k,\eta_1,l_1)\right| \left|\sqrt{-\frac{\partial_tM_3}{M_3}}\mathcal{F}(\wtb \langle \nabla\rangle^{s+1}\Theta_0)(t,0,\eta,l)\right|\d \eta \d\eta_1 \d l \d l_1 \d t\\
		&+\nu\|\wtb U^2_{\neq}\|_{L^2H^s}\|\nabla_LF_{\neq}\|_{L^2H^s}\|\wtb\Theta_0\|_{L^{\infty} H^{s+1}}\\
		\lesssim&\|(pq)^{\frac{1}{2}}U^2_{\neq}\|_{L^2H^s}\|(pq)^{\frac{1}{2}}\Theta_{\neq}\|_{L^2H^s}\|\wtb \Theta_0\|_{L^\infty H^{s+1}}+\|(pq)^{\frac{1}{2}}U^2_{\neq}\|_{L^2H^s}\|\nabla_Lm^{\frac{1}{2}}F_{\neq}\|_{L^2H^s}\|\wtb \Theta_0\|_{L^{\infty}H^{s+1}}\\
		&+\nu^{-\frac{2}{3}}|\ln(\nu)|^{1+a_0}\left\|\sqrt{-\frac{\partial_tM_3}{M_3}}m^{\frac{1}{2}}G_{\neq}\right\|_{L^2H^s}\|M m^{1/2} p^{1/2} \wtb \Theta_{\neq}\|_{L^{\infty}H^s}\left\|\sqrt{-\frac{\partial_tM_3}{M_3}}\wtb \Theta_0\right\|_{L^2H^{s+1}}\\
		&+\nu\|\wtb U^2_{\neq}\|_{L^2H^s}\|\nabla_LF_{\neq}\|_{L^2H^s}\|\wtb\Theta_0\|_{L^{\infty} H^{s+1}}\\
		\lesssim&\nu^{-\frac{1}{6}}\epsilon\nu^{-\frac{1}{2}}\epsilon^2+\nu^{-\frac{1}{6}}\epsilon\nu^{-\frac{1}{2}}\epsilon^2+\nu^{-\frac{2}{3}}|\ln(\nu)|^{1+a_0}\epsilon^3+\nu\epsilon\nu^{-\frac{5}{6}}\epsilon^2\lesssim\nu^{-\frac{2}{3}}|\ln(\nu)|^{1+a_0}\epsilon^3.
	\end{align*}
    This completes the improvement of \eqref{bs-u20}.
    
	\subsection{Energy estimates on $V^1_0$}
	\subsubsection{Energy estimates on $\wtb V^1_0$}
	In this subsection, we improve\eqref{bs-v1}. For $0\leq s_1\leq s$, the energy estimate reads
	\begin{align*}
		&\|\wtb \nabla^{s_1} M V^1_0(T)\|^2_{ L^2}+\nu\| \nabla^{s_1+1} \wtb M V^1_0\|^2_{L^2L^2}+\left\|\sqrt{-\frac{\partial_tM_3}{M_3}}\wtb \nabla^{s_1}M V^1_0\right\|^2_{L^2L^2}\\
		\lesssim&\|\wtb \nabla^{s_1} V^1_{0,in}\|^2_{ L^2 } +\mathcal{L} -\mathcal{NL}_0-\mathcal{NL}_{\neq},
	\end{align*}
    where we denote
    \begin{align*}
        &\mathcal{L}=\beta^{-1}\int_{0}^{T}\langle \wtb \nabla^{s_1}(\mu-\kappa)\Delta M\Theta_0, \wtb \nabla^{s_1} M V^1_0\rangle_{L^2} dt,\\
        &\mathcal{NL}_{0}=\int_{0}^{T}\langle \wtb \nabla^{s_1} M(U_{0}\cdot\nabla V^1_{0}), \wtb \nabla^{s_1} M V^1_0\rangle_{L^2} dt,\\
        &\mathcal{NL}_{\neq}=\int_{0}^{T}\langle \wtb \nabla^{s_1} M (U_{\neq}\cdot\nabla_L V^1_{\neq})_0, \wtb \nabla^{s_1} M V^1_0\rangle_{L^2} dt.
    \end{align*}
    For linear error, the bootstrap assumptions gives that
    \begin{align*}
        \mathcal{L}\lesssim \nu\|\nabla\wtb\Theta_0\|_{L^2H^s}\|\nabla\wtb V^1_0\|_{L^2H^s}\lesssim C^{-1}_0(10C_0\epsilon)^2.
    \end{align*}
	By interpolation, it suffices to treat this when $s_1=0$ and $s_1=s$, and the low order energy estimates can be derived by a same procedure. For $s_1=s$, we decompose $\mathcal{NL}_0$ into three terms
    \begin{align*}
		\mathcal{NL}_0=&\int_{0}^{T}\langle \wtb  M (U^2_0 \nabla^{s} \partial_Y  V^1_0), \wtb \nabla^{s} M V^1_0\rangle_{L^2} dt\\
		&+\sum_{s_2=1}^{s}\int_{0}^{T}\langle \wtb  M (\nabla^{s_1}U^2_0  \nabla^{s-s_2}\partial_YV^1_0 ), \wtb \nabla^{s} M V^1_0\rangle_{L^2} dt\\
        &+ \int_{0}^{T}\langle \wtb  M ( \nabla^{s}U^3_0  \partial_Z  V^1_0), \wtb \nabla^{s} M V^1_0\rangle_{L^2} \d t \\
        &+\int_{0}^{T}\langle \wtb  M ([\nabla^{s},\partial_Z V^1_0]U^3_0 ), \wtb \nabla^{s} M V^1_0\rangle_{L^2} \d t =:T^2_1+T^2_2+T^2_3+T^2_4.
	\end{align*}
	Using \eqref{est-2l} and bootstrap assumptions, we derive that
	\begin{align*}
		T^2_1\lesssim&\|\wtb U^2_0\|_{L^2L^\infty}\|\partial_Y \wtb V^1_0\|_{L^2H^s}\|\wtb V^1_0\|_{L^{\infty}H^s}\\
		\lesssim&\nu^{-7/12-\delta}\epsilon^3.
	\end{align*}
	We treat with $s_2=s$ for simplicity. Substituting equation of $\Theta_0$ into $T^2_2$, we deduce that
	\begin{align*}
		T^2_2=&\int_{0}^{T}\langle \wtb  M (\nabla^{s} (\partial_t \Theta_0 -\nu \Delta \Theta_0 + (U\cdot\nabla_L \Theta)_0) \partial_Y V^1_0 ), \wtb \nabla^{s} M V^1_0\rangle_{L^2} dt\\
		=&\langle \wtb  M (\nabla^{s} \Theta_0  \partial_Y V^1_0 ), \wtb \nabla^{s} M V^1_0\rangle_{L^2}\big|_{t=0}^{t=T}\\
		&-2\int_{0}^{T}\langle \wtb  \partial_tM (\nabla^{s}  \Theta_0 \partial_Y V^1_0 ), \wtb \nabla^{s} M V^1_0\rangle_{L^2} dt \\
		& -\int_{0}^{T}\langle \wtb  M (\nabla^{s}  \Theta_0 \partial_Y (\nu\Delta V^1_0-(U\cdot\nabla_L V^1)_0) ), \wtb \nabla^{s} M V^1_0\rangle_{L^2} dt\\
		&-\int_{0}^{T}\langle \wtb  M (\nabla^{s}  \Theta_0 \partial_Y  V^1_0 ), \wtb \nabla^{s} M (\nu\Delta V^1_0-(U\cdot\nabla_L V^1)_0) \rangle dt\\
		&+ \int_{0}^{T}\langle \wtb  M (\nabla^{s} ((U\cdot\nabla_L \Theta)_0 - \nu \Delta \Theta_0 ) \partial_Y V^1_0 ), \wtb \nabla^{s} M V^1_0\rangle_{L^2} dt=\sum_{j=1}^{5}T^2_{2j}.
	\end{align*}
	The bootstrap assumption implies that
	\begin{align*}
		T^2_{21}+T^2_{22}\lesssim&\|\wtb \Theta_0\|_{L^{\infty}H^s}\|\wtb  V^1_0\|_{L^\infty H^s}^{2}\\
		&+	\|\wtb \Theta_0\|_{L^\infty H^s}\|\nabla \wtb V^1_0\|_{L^2 H^{s-1}}\left\|\sqrt{-\frac{\partial_tM_3}{M_3}}\wtb M V^1_0\right\|^2_{L^2H^{s}}\\
		\lesssim &\epsilon^3+\epsilon\nu^{-1/2}\epsilon^2.
	\end{align*}
	Turning to $T^2_{24}$, integration by parts gives that
	\begin{align*}
		T^2_{24}\lesssim& \nu \|\nabla \wtb \Theta_0\|_{L^2H^s}\|\wtb V^1_0\|_{L^\infty H^s}\|\nabla \wtb V^1_0\|_{L^2H^s}\\
		&+\|\wtb (U^{2}_0, U_0^3 ,\Theta_0)\|^2_{L^\infty H^s}\|\nabla \wtb V^1_0\|^2_{L^2H^s}\\
		&+\|\wtb \Theta_0\|_{L^\infty H^s}\|\wtb V^1_0\|_{L^\infty H^s}\|\wtb U_{\neq}\|_{L^2H^s}\|\nabla_L\wtb(U^1_{\neq},\Theta_{\neq})\|_{L^2H^s}\\
        \lesssim& \epsilon^3 +\epsilon^2\nu^{-1}\epsilon^2+\epsilon^2\nu^{-1/6}\epsilon\nu^{-1/2}\epsilon |\ln{\nu}|^{2+2a_0}\lesssim \nu^{-2/3} |\ln{\nu}|^{2+2a_0} \epsilon^3.
	\end{align*}
	The remaining terms $T^2_{23}$ and $T^2_{25}$ can be treated similarly. For $T^2_3$, by integration by parts and \eqref{est-3}, we obtain
	\begin{align*}
		T^2_3\lesssim\|\wtb U^3_0\|_{L^2W^{s-1,\infty}}\|\wtb V^1_0\|_{L^\infty H^s}\|\nabla \wtb V^1_0\|_{L^2H^s}\lesssim \nu^{-\delta}\epsilon^2\nu^{-1/2}\epsilon.
	\end{align*}
    Commutator estimate gives
    \begin{align*}
        T^2_4 \lesssim \nm{\wtb U^3_0}{{L^2W^{s-1,\infty}}}  \nm{\wtb \partial_Z V^1_0}{L^2 H^{s}} \nm{\wtb V^1_0}{L^{\infty} H^s} \lesssim \nu^{-\delta-1/2}\epsilon^3.
    \end{align*}
    \color{black}
	Since the treatment of $\mathcal{NL}_{\neq}$ is similar to that of  $T^2_{24}$, we omit it and finish the improvement of \eqref{bs-v1} .
    For $s_1=0$, using \eqref{est-2l} and bootstrap assumptions, we derive that
	Using \eqref{est-2l} and bootstrap assumptions, we derive that
	\begin{align*}
		\mathcal{NL}_0=&\int_{0}^{T}\langle \wtb  M (U_0 \cdot \nabla  V^1_0), \wtb  M V^1_0\rangle_{L^2} dt \\ \lesssim&\|\wtb (U^2_0, U^3_0)\|_{L^2L^\infty}\| \nabla \wtb V^1_0\|_{L^2H^s}\|\wtb V^1_0\|_{L^{\infty}H^s}\\
		\lesssim&\nu^{-\frac{1}{12} -\delta}\epsilon\nu^{-\frac{1}{2}}\epsilon^2=\nu^{-7/12 -\delta}\epsilon^3.
	\end{align*}
	Since the treatment of $\mathcal{NL}_{\neq}$ is similar to that of  $T^2_{24}$, we omit it and finish the improvement of \eqref{bs-v1} .
    \color{black}

    \subsubsection{ Energy estimate for $\nabla V_0^1$ }   
    In this subsection, we improve \eqref{bs-v1h}. For $0\leq s_1\leq s$, the energy estimate reads
    \begin{align*}
    	&\| \nabla^{s_1+1} M V^1_0(T)\|^2_{ L^2}+ \nu\| \nabla^{s_1+2}  M V^1_0\|^2_{L^2L^2}+\left\|\sqrt{-\frac{\partial_tM_3}{M_3}}\nabla^{s_1+1} M V^1_0\right\|^2_{L^2L^2}\\
    	\lesssim&\| \nabla^{s_1+1} V^1_{0,in}\|^2_{ L^2 } +\mathcal{L}^h -\mathcal{NL}^h_0-\mathcal{NL}^h_{\neq}, 
    \end{align*}
    where we denote
    \begin{align*}
    &\mathcal{L}=\beta^{-1}\int_{0}^{T}\langle \nabla^{s_1+1}(\mu-\kappa)\Delta M\Theta_0, \nabla^{s_1+1} M V^1_0\rangle_{L^2} dt\\
        &\mathcal{NL}^{h}_0=\int_{0}^{T}\langle  \nabla^{s_1+1} M(U_{0}\cdot\nabla V^1_{0}), \nabla^{s_1+1} M V^1_0\rangle_{L^2} dt,\\
        &\mathcal{NL}^{h}_{\neq}=\int_{0}^{T}\langle \nabla^{s_1+1} M (U_{\neq}\cdot\nabla_L V^1_{\neq})_0,  \nabla^{s_1+1} M V^1_0\rangle_{L^2} dt.
    \end{align*}
    Without loss of generality, we only estimate it when $s_1=s$. 
    For linear error, integration by parts implies that 
    \begin{align*}
        \mathcal{L}\lesssim \nu\|\nabla\Theta_0\|_{L^2H^{s+1}}\|\nabla^2 V^1_0\|_{L^2H^s}\lesssim C^{-1}_0(10C_0\nu^{-1/6}\epsilon)^2.
    \end{align*}
    For $\mathcal{NL}^h_0$, commutator and Strichartz estimates yield that
    \begin{align*}
        \mathcal{NL}^{h}_0 \lesssim & \int_0^T \left| \left \langle M \left[\nabla^{s+1} , \left( U_0 \cdot  \nabla \right) \right]  V_0^1 , M \nabla ^{s+1} V_0^1 \right \rangle _{L^2} \right| \d t \\
        &+\int_0^T\left|\langle M (U_0\cdot\nabla^{s+2}V^1_0),M\nabla^{s+1}V^1_0\rangle_{L^2}\right|dt\\
        \lesssim & \left \| \nabla U_0^{2,3} \right \|_{L^{\infty} H^{s}} \left\|  \nabla V_0^1 \right\|_{L^2 H^s} \left\|  \nabla^{s+1} V_0^1 \right\|_{L^2 L^2}+\|U^{2,3}_0\|_{L^2L^\infty}\|\nabla^2V^1_0\|_{L^2H^s}\|\nabla V^1_0\|_{L^\infty H^s} \\
        \lesssim & \epsilon^3 \nu^{-1/2-1/2} +\nu^{-1/12-\delta}\epsilon\nu^{-1/2}(\nu^{-1/6}\epsilon)^2\\
        =&\epsilon \nu^{-2/3} (\epsilon \nu^{-1/6})^2,
    \end{align*}
    where we employ $\nu^{1/2} \nm{\wtb \nabla V_0^1 }{\dot{H^s}} \leqslant \epsilon $ as in the $\dot{H}^s$ estimate of $\wtb V_0^1$.
    \color{black}    
   	Using lemma \ref{est-fun}, we conclude that 
   	\begin{align*}
   		\mathcal{NL}^h_{\neq}=&\int_{0}^{T}\langle |\nabla|^{s_1+1} M (U_{\neq}\cdot\nabla_L V^1_{\neq})_0,  |\nabla|^{s_1+1} M V^1_0\rangle_{L^2} dt\\
   		\lesssim&\|p^{\frac{1}{2}}\partial_X(U^1_{\neq},\Theta_{\neq})\|_{L^2H^s} \|\partial_X(U^1_{\neq},\Theta_{\neq})\|_{L^2H^s}\|\nabla V^1_0\|_{L^\infty H^s}\\
   		&+(\|p^{\frac{1}{2}}U^2_{\neq}\|_{L^2H^s} \|\partial_Y^LU^1_{\neq}\|_{L^2H^s}+\|U^2_{\neq}\|_{L^2H^s}\|\partial_Y^L p^{\frac{1}{2}}U^1_{\neq}\|_{L^2H^s})\|\nabla  V^1_0\|_{L^\infty H^s}\\
   		&+(\|p^{\frac{1}{2}}U^3_{\neq}\|_{L^2H^s} \|\partial_ZU^1_{\neq}\|_{L^2H^s}+\|U^3_{\neq}\|_{L^2H^s}\|\partial_Z p^{\frac{1}{2}}U^1_{\neq}\|_{L^2H^s})\|\nabla  V^1_0\|_{L^\infty H^s}\\
   		\lesssim&\nu^{-1/2} |\ln{\nu}|^{2+2a_0} \epsilon(\nu^{-1/6}\epsilon)^2+(\nu^{-1/6}\epsilon\nu^{-1/2}\epsilon+\epsilon\nu^{-5/6}\epsilon)|\ln{\nu}|^{1+a_0} \nu^{-1/6}\epsilon\\
        &+  (\nu^{-1/2}\epsilon\nu^{-1/6}\epsilon+\nu^{-1/6}\epsilon\nu^{-1/2}\epsilon ) |\ln{\nu}|^{2+2a_0}  \nu^{-1/6}\epsilon\\
   		\lesssim&\nu^{-2/3} \epsilon|\ln{\nu}|^{1+a_0} (\nu^{-1/6}\epsilon)^2,
   	\end{align*}
   	and finish the improvement of \eqref{bs-v1h}.
    
	\section{Energy estimates on nonzero modes}
	\subsection{Energy estimates on $(G_{\neq},F_{\neq})$ system}
	In this subsection, we improve \eqref{bs-gf}. The energy estimate of $(G_{\neq},F_{\neq})$ system reads
	\begin{align*}
		&\|m^{\frac{1}{2}}M(G_{\neq},F_{\neq})(T)\|^2_{H^s}+ \nu \|m^{\frac{1}{2}}M\nabla_L(G_{\neq},F_{\neq})\|^2_{L^2H^s}+\sum_{i=1,2,3}\left \|\sqrt{-\frac{\partial_t M_i}{M_i}}m^{\frac{1}{2}}M(G_{\neq},F_{\neq})\right\|^2_{L^2H^s} \\
		\lesssim& \|(G_{\neq,in},F_{\neq,in})\|^2_{H^s}- \sum_{i=1}^{3}T^{3}_i,
	\end{align*}
	where nonlinear terms are given by
	\begin{align*}
		T^3_1=&\int_{0}^{T}\langle m^{\frac{1}{2}}Mp(U\cdot\nabla_L U^2)_{\neq},m^{\frac{1}{2}}MG\rangle_{H^s} dt,\\
		T^3_2=&\int_{0}^{T}\langle m^{\frac{1}{2}}M\partial_Y^L(\partial^L_i U^j \partial^L_j U^i)_{\neq},m^{\frac{1}{2}}MG\rangle_{H^s} dt,\\
		T^3_3=&\int_{0}^{T}\langle m^{\frac{1}{2}}M(pq)^{\frac{1}{2}}(U\cdot\nabla_L \Theta)_{\neq},m^{\frac{1}{2}}MF\rangle_{H^s} dt.
	\end{align*}
	We introduce the notation, for $j\in\{1,2,3\}$ and $s_3,s_4\in\{0,\neq \}$,
	\begin{align*}
		T^3_1(j,s_3,s_4):=&\int_{0}^{T}\langle m^{\frac{1}{2}}Mp(U^j_{s_3}\partial_j^L U^2_{s_4}),m^{\frac{1}{2}}MG\rangle_{H^s} dt.
	\end{align*}
	It follows from integration by parts and lemma \ref{est-fun} that
	\begin{align*}
		T^3_1(1,0,\neq)\lesssim&\|\nabla U^1_0\|_{L^\infty H^s}\|M\partial_XU^2_{\neq}\|_{L^2H^s}\|\nabla_L m^{\frac{1}{2}}MG_{\neq}\|_{L^2H^s}\\
		&+\|U^1_0\|_{L^\infty H^s}\|M\partial_X\nabla_LU^2_{\neq}\|_{L^2H^s}\|\nabla_L m^{\frac{1}{2}}MG_{\neq}\|_{L^2H^s}\\
		\lesssim&\nu^{-1/6}\epsilon^2\nu^{-1/2}\epsilon+\epsilon\nu^{-1/6}\epsilon\nu^{-1/2}\epsilon=\nu^{-2/3}\epsilon^3.
	\end{align*}
    For $T^{3}_1(1,\neq,\neq)$, using \eqref{ineq-m31}, \eqref{ineq-m32}, and \eqref{prop-m3}, one obtain that
	\begin{align*}
		&T^3_1(1,\neq,\neq)\\\lesssim&\|U^1_{\neq}\|_{L^2H^s}\|\partial_XMG_{\neq}\|_{L^2H^s}\|m^{\frac{1}{2}}MG_{\neq}\|_{L^\infty H^{s}}\\
		&+\|p^{\frac{1}{2}}U^1_{\neq}\|_{L^2H^s}\|\partial_Xp^{\frac{1}{2}}MU^2_{\neq}\|_{L^2H^s}\|m^{\frac{1}{2}}MG_{\neq}\|_{L^\infty H^{s}}\\
		&+\sum_{k,k_1\neq0}\int_{0}^{T}\int_{\mathbb{R}^4}|\mathcal{F}(p\langle\nabla\rangle^s U^1_{\neq})(t,k-k_1,\eta-\eta_1,l-l_1)|\\
		&\quad\quad\times|\langle t\rangle^{-1}k_1\mathcal{F}(\langle \nabla\rangle p^{\frac{1}{2}}U^2_{\neq})(t,k_1,\eta_1,l_1)||\mathcal{F}(m^{\frac{1}{2}}M\langle \nabla\rangle^{s}G_{\neq})(t,k,\eta,l)|\d\eta \d\eta_1 \d l \d l_1 \d t\\
		&+|\ln(\nu)|^{1+a_0}\sum_{k,k_1\neq0}\int_{0}^{T}\int_{\mathbb{R}^4} \langle t\rangle \left|\mathcal{F}(\langle\nabla\rangle^{2+\frac{a_0(2+a_0)}{2}} p^{\frac12} U^1_{\neq})(t,k-k_1,\eta-\eta_1,l-l_1)\right|\\
		&\quad\quad\times\left|\sqrt{-\frac{\partial_tM_3}{M_3}}\mathcal{F}(\langle \nabla \rangle^sp^{\frac{1}{2}}MU^2_{\neq})(t,k_1,\eta_1,l_1)\right|\\
		&\quad\quad\times\left|\sqrt{-\frac{\partial_tM_3}{M_3}}\mathcal{F}(m^{\frac{1}{2}}M \langle \nabla\rangle^{s}G_{\neq})(t,k,\eta,l)\right| \d\eta \d\eta_1 \d l \d l_1 \d t\\
		&+\nu\|\wtb U^2_{\neq}\|_{L^2H^s}\|\nabla_LF_{\neq}\|_{L^2H^s}\|\wtb\Theta_0\|_{L^{\infty} H^{s+1}}\\
		\lesssim&
		\|U^1_{\neq}\|_{L^2H^s}\|\nabla_L m^{\frac{1}{2}}MG_{\neq}\|_{L^2H^s}\|m^{\frac{1}{2}}MG_{\neq}\|_{L^\infty H^{s}}+\|p^{\frac{1}{2}}U^1_{\neq}\|_{L^2H^s}\|m^{\frac{1}{2}}MG_{\neq}\|_{L^2H^s}\|m^{\frac{1}{2}}MG_{\neq}\|_{L^\infty H^{s}}\\
		&+\|m^{\frac{1}{2}}pU^1_{\neq}\|_{L^2H^s}\|m^{\frac{1}{2}}MG_{\neq}\|_{L^2H^s}\|m^{\frac{1}{2}}MG_{\neq}\|_{L^\infty H^{s}}\\
        &+\nu^{-\frac{2}{3}}|\ln(\nu)|^{1+a_0}\|M m^{1/2}p^{1/2}U^1_{\neq}\|_{L^\infty H^s}\left\|\sqrt{-\frac{\partial_tM_3}{M_3}}m^{\frac{1}{2}}MG_{\neq}\right\|^2_{L^2H^s}\\
		&+\nu\|p U^1_{\neq}\|_{L^2H^s}\|M\partial_XU^2_{\neq}\|_{L^2H^s}\|m^{\frac{1}{2}}MG_{\neq}\|_{L^{\infty} H^{s}}\\
		\lesssim&|\ln(\nu)|^{1+a_0}\nu^{-\frac{1}{6}}\epsilon\nu^{-\frac{1}{2}}\epsilon^2+\nu^{-1/2}|\ln(\nu)|^{1+a_0}\epsilon\nu^{-1/6}\epsilon+\nu^{-\frac{1}{2}}|\ln(\nu)|^{1+a_0}\epsilon\nu^{-\frac{1}{6}}\epsilon^2\\
        &+\nu^{-\frac{2}{3}}|\ln(\nu)|^{2+2a_0}\epsilon^3+\nu\nu^{-5/6}|\ln(\nu)|^{1+a_0}\epsilon^3\\
		\lesssim&\nu^{-\frac{2}{3}}|\ln(\nu)|^{2+2a_0}\epsilon^3.
	\end{align*}
	The incompressible condition and integration by parts imply that
	\begin{equation}\label{ineq-t41}
	    \begin{aligned}
		&T^3_1(2,0,\neq)+T^3_1(2,\neq,0)\\
		\lesssim&\|\wtb (U^2_0,U^3_0)\|_{L^{\infty} H^{s+1}}\|\partial_Y^LMU^2_{\neq}\|_{L^2H^s}\|m^{\frac{1}{2}}MG_{\neq}\|_{L^2H^s}+\|\nabla\langle \nabla\rangle U^2_{0}\|_{L^\infty H^s}\|m^{\frac{1}{2}}MG_{\neq}\|^2_{L^2H^s}\\
		&+\|\langle\nabla\rangle U^2_{0}\|_{L^\infty H^s}\|m^{\frac{1}{2}}\nabla_LMG_{\neq}\|_{L^2 H^s}\|m^{\frac{1}{2}}MG_{\neq}\|_{L^2H^s}\\
		&+\|m^{\frac{1}{2}}Mp^{\frac{1}{2}}U^2_{\neq}\|_{L^2H^s}\|\partial_ZU^3_0\|_{L^\infty H^{s+1}}\|\nabla_Lm^{\frac{1}{2}}MG_{\neq}\|_{L^2H^s}\\
		&+\|MU^2_{\neq}\|_{L^2H^s}\|\partial_ZU^3_0\|_{L^\infty H^{s+1}}\|\nabla_Lm^{\frac{1}{2}}MG_{\neq}\|_{L^2H^s}\\
		\lesssim&2\epsilon(\nu^{-1/6}\epsilon)^2+\epsilon\nu^{-1/2}\epsilon\nu^{-1/6}\epsilon+2\epsilon^2\nu^{-1/2}\epsilon\lesssim\nu^{-2/3}\epsilon^3.
	\end{aligned}
	\end{equation}
	The estimates of $T^3_1(2,\neq,\neq)$ is similar to that of $T^1_3$. We omit the details and conclude that
	\begin{align*}
		&T^3_1(2,\neq,\neq)\\
        \lesssim&(\|p U^2_{\neq}\|_{L^2H^s}\|M\partial_Y^LU^2_{\neq}\|_{L^2H^s}+\|M p^{\frac12} U^2_{\neq}\|_{L^2H^s}\|p^{\frac12} \partial_Y^L U^2_{\neq}\|_{L^2H^s})\|m^{\frac{1}{2}}MG_{\neq}\|_{L^\infty H^s}\\
		&+\|p^{\frac{1}{2}}U^2_{\neq}\|_{L^2H^{5/2+}}\|\nabla_Lm^{\frac{1}{2}}MG_{\neq}\|_{L^2H^s}\|m^{\frac{1}{2}}M G_{\neq}\|_{L^{\infty}H^{s}}\\
		&+|\ln(\nu)|^{1+a_0}\left\|\sqrt{-\frac{\partial_tM_3}{M_3}}m^{\frac{1}{2}}MG_{\neq}\right\|_{L^2H^s}\nu^{-\frac{2}{3}}\|m^{\frac12}M p^{\frac12} \partial_Y^L U^2_{\neq}\|_{L^{\infty}H^{7/2+}}\left\|\sqrt{-\frac{\partial_tM_3}{M_3}}m^{\frac{1}{2}}M G_{\neq}\right\|_{L^2H^{s}}\\
		&+\nu\| U^2_{\neq}\|_{L^2H^s}\|\nabla_LG_{\neq}\|_{L^2H^s}\|m^{\frac{1}{2}}M G_{\neq}\|_{L^{\infty} H^{s}}\\
		\lesssim&(\nu^{-1/2}\epsilon\nu^{-1/6}\epsilon+\nu^{-1/6}\epsilon\nu^{-1/2}\epsilon)\epsilon+\nu^{-1/6}\epsilon\nu^{-1/2}\epsilon^2+|\ln(\nu)|^{1+a_0}\epsilon\nu^{-2/3}\epsilon^2+\nu \epsilon\nu^{-5/6}\epsilon^2\\
		\lesssim&\nu^{-2/3}|\ln(\nu)|^{1+a_0}\epsilon^3.
	\end{align*}
	For $T^3_1(3,\cdot,\cdot)$, arguing as in the estimates of $T^3_1(1,\cdot,\cdot)$, we deduce that
	\begin{align*}
		&T^3_1(3,0,\neq)+T^3_1(3,\neq,0)\\
		\lesssim&(\|U^3_0\|_{L^\infty H^{s+1}}\|\partial_ZMU^2_{\neq}\|_{L^2H^s}+\|U^3_0\|_{L^\infty H^{s+1}}\|\partial_Zp^{\frac{1}{2}}m^{\frac{1}{2}}MU^2_{\neq}\|_{L^2H^s})\|\nabla_Lm^{\frac{1}{2}}MG_{\neq}\|_{L^2H^s}\\
		&+(\|m^{\frac{1}{2}}Mp^{\frac{1}{2}}U^3_{\neq}\|_{L^2H^s}\|\partial_ZU^2_0\|_{L^\infty H^{s+1}}+\|MU^3_{\neq}\|_{L^2H^s}\|\partial_ZU^2_0\|_{L^\infty H^{s+1}})\|\nabla_L m^{\frac{1}{2}}MG_{\neq}\|_{L^2H^s}\\
		\lesssim&(\epsilon^2+\epsilon\nu^{-1/6}\epsilon)\nu^{-1/2}\epsilon+2\nu^{-1/6}|\ln(\nu)|^{1+a_0}\epsilon^2\nu^{-1/2}\epsilon\lesssim\nu^{-2/3}|\ln(\nu)|^{1+a_0}\epsilon^3,
	\end{align*}
	and
	\begin{align*}
		T^3_1(3,\neq,\neq)\lesssim&\|U^3_{\neq}\|_{L^2H^s}\|\partial_ZMG_{\neq}\|_{L^2H^s}\|m^{\frac{1}{2}}MG_{\neq}\|_{L^\infty H^{s}}\\
		&+\|p^{\frac{1}{2}}U^3_{\neq}\|_{L^2H^s}\|\partial_Zp^{\frac{1}{2}}MU^2_{\neq}\|_{L^2H^s}\|m^{\frac{1}{2}}MG_{\neq}\|_{L^\infty H^{s}}\\
		&+\|m^{\frac{1}{2}}pU^3_{\neq}\|_{L^2H^s}\|m^{\frac{1}{2}}MG^2_{\neq}\|_{L^2H^3}\|m^{\frac{1}{2}}MG_{\neq}\|_{L^\infty H^{s}}\\
		&+\nu^{-\frac{2}{3}}|\ln(\nu)|^{1+a_0} \|Mm^{\frac12} p^{\frac12} U^3_{\neq}\|_{L^\infty H^{7/2+}}\left\|\sqrt{-\frac{\partial_tM_3}{M_3}}m^{\frac{1}{2}}MG_{\neq}\right\|^2_{L^2H^s}\\
		&+\nu\|p U^3_{\neq}\|_{L^2H^s}\|M\partial_ZU^2_{\neq}\|_{L^2H^s}\|m^{\frac{1}{2}}MG_{\neq}\|_{L^{\infty} H^{s}}\\
		\lesssim&\nu^{-\frac{1}{6}}|\ln(\nu)|^{1+a_0}\epsilon\nu^{-\frac{1}{2}}\epsilon^2+\nu^{-1/2}|\ln(\nu)|^{1+a_0}\epsilon\nu^{-1/6}\epsilon+\nu^{-\frac{1}{2}}|\ln(\nu)|^{1+a_0}\epsilon\nu^{-\frac{1}{6}}\epsilon^2\\
        &+\nu^{-\frac{2}{3}}|\ln(\nu)|^{2+2a_0}\epsilon^3+\nu\nu^{-5/6}|\ln(\nu)|^{1+a_0}\epsilon^3\\
		\lesssim&\nu^{-\frac{2}{3}}|\ln(\nu)|^{2+2a_0}\epsilon^3.
	\end{align*}
	
	Turning to the nonlinear pressure term $T^4_2$, we introduce the notation for $(i,j)\in\{1,2,3\}$ and $s_3,s_4\in\{0,\neq\}$,
	\begin{equation*}
		T^3_2(i,j,s_3,s_4)=\int_{0}^{T}\langle m^{\frac{1}{2}}M\partial_Y^L(\partial^L_i U^j_{\neq} \partial^L_j U^i_{\neq}),m^{\frac{1}{2}}MG\rangle_{H^s} dt,\\
	\end{equation*}
	For $i=j\in\{1,3\}$, it follows from lemma \eqref{est-fun} that
	\begin{align*}
		T^3_2(i,j)\lesssim&\|\partial_XU^1_{\neq}\|_{L^2H^s}\|\partial_XMU^1_{\neq}\|_{L^\infty H^s}\|\nabla_Lm^{\frac{1}{2}}MG_{\neq}\|_{L^2H^s}\\
		&+\|\partial_ZU^1\|_{L^\infty H^s}\|\partial_XMU^3_{\neq}\|_{L^2H^s}\|\nabla_Lm^{\frac{1}{2}}MG_{\neq}\|_{L^2H^s}\\
		&+\|\partial_ZU^3\|_{L^\infty H^s}\|\partial_ZMU^3_{\neq}\|_{L^2H^s}\|\nabla_Lm^{\frac{1}{2}}MG_{\neq}\|_{L^2H^s}\\
		\lesssim&3|\ln(\nu)|^{1+a_0}\epsilon\nu^{-1/6}|\ln(\nu)|^{1+a_0}\epsilon\nu^{-1/2}\epsilon\lesssim\nu^{-2/3}|\ln(\nu)|^{2+2a_0}\epsilon^3.
	\end{align*}
	The estimates of $i\in\{1,2,3\},j=2$ are obtained in $T^3_1$. The treatment of $T^3_3(\cdot,\neq,\neq)$ is exact the same as that of $T^1$. The bootstrap assumption and \eqref{prom3} give that
	\begin{align*}
		&T^3_3(\cdot,0,\neq)+T^3_3(\cdot,\neq,0)\\
		\lesssim&\|\wtb U^1_0\|_{L^\infty H^{s}}\|\partial_Xq^{\frac{1}{2}}M\Theta_{\neq}\|_{L^2H^s}\|\nabla_Lm^{\frac{1}{2}}MF_{\neq}\|_{L^2H^s}\\
		&+(\|\wtb U^2_0\|_{L^\infty H^{s+1}}\| m^{\frac{1}{2}}MF_{\neq}\|_{L^2H^s}+
		\|Mq^{\frac{1}{2}}U^2_{\neq}\|_{L^2H^s}\|\wtb \Theta_0\|_{L^\infty H^{s+1}})\|\nabla_Lm^{\frac{1}{2}}MF_{\neq}\|_{L^2H^s}\\
		&+ (\|\wtb U^3_0\|_{L^\infty H^{s}}\| qM\Theta_{\neq}\|_{L^2H^s}+\|Mq^{\frac{1}{2}}U^3_{\neq}\|_{L^2H^s}\|\wtb \Theta_0\|_{L^\infty H^{s+1}})\|\nabla_Lm^{\frac{1}{2}}MF_{\neq}\|_{L^2H^s}\\
		\lesssim&\epsilon\nu^{-1/6}\epsilon\nu^{-1/2}\epsilon+(\epsilon\nu^{-1/6}\epsilon+\epsilon^2)\nu^{-1/2}\epsilon+(\epsilon\nu^{-1/6}\epsilon+\nu^{-1/6}|\ln(\nu)|^{1+a_0}\epsilon^2)\nu^{-1/2}\epsilon\lesssim\nu^{-2/3}|\ln(\nu)|^{1+a_0}\epsilon^3.
	\end{align*}
	This complete the improvement of \eqref{bs-gf}.
	
  	\subsection{Energy estimates on $W_{\neq}$}
  	In this subsection, we improve \eqref{bs-w}. The energy estimate reads
  	
    \begin{align*}
		&\|m^{\frac{1}{2}}MW_{\neq}(T)\|^2_{H^s}+ \nu \|m^{\frac{1}{2}}M p^{1/2} W_{\neq}\|^2_{L^2H^s}+\sum_{i=1,2,3}\left \|\sqrt{-\frac{\partial_t M_i}{M_i}}m^{\frac{1}{2}}M W_{\neq}\right\|^2_{L^2H^s} \\
		\lesssim& \|W_{\neq,in}\|^2_{H^s}+LP- \sum_{i=1}^{3}T^5_i,
	\end{align*}
	where we denote
	\begin{align*}
    LP=&\int_{0}^{T}\left \langle m^{\frac{1}{2}}M\left(\frac{2\partial_{XZ}q^{\frac{1}{2}}}{p^{\frac{3}{2}}}G -\frac{\partial_{XZ}}{(pq)^{\frac{1}{2}}}G-\frac{(\partial_tp)\partial_{YZ}^L}{p^{\frac{3}{2}}q^{\frac{1}{2}}}G\right)_{\neq},m^{\frac{1}{2}}MW_{\neq}\right\rangle_{H^s}dt,\\
		T^5_1=&\int_{0}^{T}\langle m^{\frac{1}{2}}M(pq)^{1/2}(U\cdot\nabla_L U^3)_{\neq},m^{\frac{1}{2}}M W\rangle_{H^s} dt,\\
        T^5_2=&-\int_{0}^{T}\langle m^{\frac{1}{2}}M \frac{\partial_Y^L \partial_Z}{q^{1/2}} p^{1/2}(U\cdot\nabla_L U^2)_{\neq},m^{\frac{1}{2}}M W\rangle_{H^s} dt,\\
		T^5_3=&\int_{0}^{T}\langle m^{\frac{1}{2}}M\partial_Z (\frac{q}{p})^{1/2}(\partial^L_i U^j \partial^L_j U^i)_{\neq},m^{\frac{1}{2}}MW\rangle_{H^s} dt,\\
		T^5_4=&-\int_{0}^{T}\langle m^{\frac{1}{2}}M\partial_Y^L\frac{\partial_{YZ}^L}{(pq)^{\frac{1}{2}}}(\partial^L_i U^j \partial^L_j U^i)_{\neq},m^{\frac{1}{2}}MW\rangle_{H^s} dt.
	\end{align*}
    For the $LP$ term, using \eqref{ineq-m31} and \eqref{prop-m3}, we deduce that
  	\begin{align*}
  		LP\lesssim&|\ln(\nu)|^{1+a_0}
  		\left\|\sqrt{-\frac{\partial_tM_3}{M_3}}m^{\frac{1}{2}}MG_{\neq}\right\|_{L^2H^s}
  		 \left\|\sqrt{-\frac{\partial_tM_3}{M_3}}m^{\frac{1}{2}}MW_{\neq}\right\|_{L^2H^s}\\
         &+\nu\|m^{\frac{1}{2}}MG_{\neq}\|_{L^2H^s}\|m^{\frac{1}{2}}MW_{\neq}\|_{L^2H^s}
         \lesssim (C_0)^{-1}(C_0|\ln(\nu)|^{1+a_0}\epsilon)^2.
  	\end{align*}
	For nonlinear terms, we introduce the notation, for $j\in\{1,2,3\}$ and $s_3,s_4\in\{0,\neq \}$,
	\begin{align*}
		T^5_1(j,s_3,s_4):=&\int_{0}^{T}\langle m^{\frac{1}{2}}M(pq)^{1/2}(U^j_{s_3}\partial_j^L U^3_{s_4}),m^{\frac{1}{2}}M W\rangle_{H^s} dt, \\
        T^5_2(j,s_3,s_4):=&\int_{0}^{T}\langle m^{\frac{1}{2}}M \frac{\partial_Y^L \partial_Z}{q^{1/2}} p^{1/2} (U^j_{s_3}\partial_j^L U^2_{s_4}),m^{\frac{1}{2}}M W\rangle_{H^s} dt.
	\end{align*}
	It follows from integration by parts and lemma \ref{est-fun} that
	\begin{align*}
		T^5_1(1,0,\neq)\lesssim&\| \wtb U^1_0\|_{L^\infty H^s}\|M\partial_X U^3_{\neq}\|_{L^2H^s}\|\nabla_L m^{\frac{1}{2}}M W_{\neq}\|_{L^2H^s}\\
		&+\|U^1_0\|_{L^\infty H^s}\|M\partial_X \wtb U^3_{\neq}\|_{L^2H^s}\|\nabla_L m^{\frac{1}{2}}M W_{\neq}\|_{L^2H^s}\\
        \lesssim&\nu^{-1/6}\epsilon^2\nu^{-1/2}\epsilon+\epsilon\nu^{-1/6}\epsilon\nu^{-1/2}\epsilon=\nu^{-2/3}\epsilon^3.
	\end{align*}
	For $T^5_1(1,\neq,\neq)$, it is clear that
    \begin{align*}
		&T^5_1(1,\neq,\neq)\\
		\lesssim&
		\left( \|q^{1/2} U^1_{\neq}\|_{L^{\infty} H^s}\|\partial_X  U^3_{\neq}\|_{L^2H^s}
		+\|U^1_{\neq}\|_{L^{\infty} H^s}\|\partial_X q^{\frac{1}{2}} U^3_{\neq}\|_{L^2H^s} \right) \|m^{\frac{1}{2}}  p^{1/2} W_{\neq}\|_{L^2 H^{s}}\\
		\lesssim& \nu^{-1/6-1/2}|\ln(\nu)|^{3+3a_0}\epsilon^3= \nu^{-\frac{2}{3}}|\ln(\nu)|^{3+3a_0}\epsilon^3.
	\end{align*}
	The incompressible condition and integration by parts imply that
	\begin{equation}\label{ineq-t41}
	    \begin{aligned}
		~&T^5_1(2,0,\neq)+T^5_1(2,\neq,0)\\
		\lesssim&\|\wtb (U^2_0,U^3_0)\|_{L^{\infty} H^{s+1}}\|\partial_Y^LMU^3_{\neq}\|_{L^2H^s}\|m^{\frac{1}{2}}MW_{\neq}\|_{L^2H^s}\\
        ~&+\|\nabla\langle \nabla\rangle U^2_{0}\|_{L^\infty H^s}\|m^{\frac{1}{2}}M q^{1/2} \partial_Y^L U^3_{\neq}\|_{L^2H^s} \nm{m^{1/2}MW_{\neq}}{L^2 H^s} \\
        ~&+\|\wtb \langle \nabla\rangle U^2_{0}\|_{L^\infty H^s}\|m^{\frac{1}{2}}M p^{1/2} \partial_Y^L U^3_{\neq}\|_{L^2H^s} \nm{m^{1/2}MW_{\neq}}{L^2 H^s} \\
		~&+\|\langle\nabla\rangle U^2_{0}\|_{L^\infty H^s}\|m^{\frac{1}{2}} M \partial_Y^L p^{1/2} q^{1/2}   U^3_{\neq}\|_{L^2H^s}\|m^{\frac{1}{2}}M W_{\neq}\|_{L^2H^s}\\
		~&+\|Mq^{\frac{1}{2}}U^2_{\neq}\|_{L^2H^s}\|\partial_Y U^3_0\|_{L^\infty H^{s}}\|\nabla_Lm^{\frac{1}{2}}MW_{\neq}\|_{L^2H^s}\\
		~&+\|MU^2_{\neq}\|_{L^2H^s}\|\partial_Y \wtb U^3_0\|_{L^\infty H^{s}}\|\nabla_Lm^{\frac{1}{2}}MW_{\neq}\|_{L^2H^s}\\
		\lesssim& |\ln{\nu}|^{2+2a_0} ( \epsilon^3 \nu^{-1/2-1/6}+\epsilon(\nu^{-1/6}\epsilon)^2+\epsilon\nu^{-1/2}\epsilon\nu^{-1/6}\epsilon+ \epsilon^3 \nu^{-1/2-1/6} +2\epsilon^2\nu^{-1/2}\epsilon)\\
        \lesssim&\nu^{-2/3}\epsilon^3.
	\end{aligned}
	\end{equation}
	The estimates of $T^5_1(2,\neq,\neq)$ is similar to that of $T^1_3$. We omit the details and conclude that
    \begin{align*}
		~&T^5_1(2,\neq,\neq)\\
        \lesssim&(\| (pq)^{1/2} U^2_{\neq}\|_{L^2H^s}\|\partial_Y^L M U^3_{\neq}\|_{L^2H^s} \|m^{\frac{1}{2}}M W_{\neq}\|_{L^\infty H^s}+\| p^{1/2} U^2_{\neq}\|_{L^2H^s}\|  q^{1/2} \partial_Y^L MU^3_{\neq}\|_{L^2H^s}\|m^{\frac{1}{2}} M W_{\neq}\|_{L^\infty H^s} \\
		&+\|p^{\frac{1}{2}} (1+q^{1/2})U^2_{\neq}\|_{L^2H^{5/2+}}\|\nabla_Lm^{\frac{1}{2}}M W_{\neq}\|_{L^2H^s}\|m^{\frac{1}{2}}M W_{\neq}\|_{L^{\infty}H^{s}}\\
		&+|\ln(\nu)|^{1+a_0}\left\|\sqrt{-\frac{\partial_tM_3}{M_3}}m^{\frac{1}{2}}M G_{\neq}\right\|_{L^2H^s}\nu^{-\frac{2}{3}} \|m^{1/2}M W_{\neq}\|_{L^{\infty}H^{7/2+}}\left\|\sqrt{-\frac{\partial_tM_3}{M_3}}m^{\frac{1}{2}}M W_{\neq}\right\|_{L^2H^{s}}\\
		&+\nu\| (1+q^{1/2}) U^2_{\neq}\|_{L^2H^s}\|\nabla_L MW_{\neq}\|_{L^2H^s}\|m^{\frac{1}{2}}M W_{\neq}\|_{L^{\infty} H^{s}}\\
		\lesssim& |\ln{\nu}|^{2+2a_0}(2\nu^{-1/6}\epsilon\nu^{-1/2}\epsilon^2+\nu^{-1/6}\epsilon\nu^{-1/2}\epsilon^2+|\ln(\nu)|^{1+a_0}\epsilon\nu^{-2/3}\epsilon^2+\nu \epsilon\nu^{-5/6}\epsilon^2)\\
		\lesssim&\nu^{-2/3}|\ln(\nu)|^{3+3a_0}\epsilon^3.
	\end{align*}
	For $T^3_1(3,\cdot,\cdot)$, arguing as in the estimates of $T^3_1(1,\cdot,\cdot)$, we deduce that
	\begin{align*}
		&T^5_1(3,0,\neq)+T^5_1(3,\neq,0)\\
		\lesssim&(\|U^3_0\|_{L^\infty H^{s+1}}\|\partial_ZMU^3_{\neq}\|_{L^2H^s}+\|U^3_0\|_{L^\infty H^{s+1}}\|\partial_Z q^{\frac{1}{2}}m^{\frac{1}{2}}MU^3_{\neq}\|_{L^2H^s})\|\nabla_Lm^{\frac{1}{2}}MW_{\neq}\|_{L^2 H^s}\\
		&+(\|m^{\frac{1}{2}}Mq^{\frac{1}{2}}U^3_{\neq}\|_{L^2H^s}\|\partial_Z U^3_0\|_{L^\infty H^{s+1}}+\|MU^3_{\neq}\|_{L^2H^s}\|\partial_ZU^3_0\|_{L^\infty H^{s+1}})\|\nabla_L m^{\frac{1}{2}}MG_{\neq}\|_{L^2H^s}\\
		\lesssim&2\epsilon^3  \nu^{-1/6-1/2} |\ln{\nu}|^{2+2a_0} +2\nu^{-1/6}|\ln(\nu)|^{1+a_0}\epsilon^2\nu^{-1/2}\epsilon\lesssim\nu^{-2/3}|\ln(\nu)|^{2+2a_0}\epsilon^3,
	\end{align*}
	and
	\begin{align*}
		T^3_1(3,\neq,\neq)\lesssim & \left( \| q^{1/2} U^3_{\neq}\|_{L^{\infty}H^s}\|\partial_Z M  U^3_{\neq}\|_{L^{\infty}H^s} +\|  U^3_{\neq}\|_{L^2H^s}\|q^{1/2}\partial_Z M  U^3_{\neq}\|_{L^2H^s} \right) \|m^{\frac{1}{2}} p^{1/2} M W_{\neq}\|_{L^2 H^{s}}\\
        \lesssim & |\ln{\nu}|^{3+3a_0} \nu^{-1/6-1/2} \epsilon^3.
	\end{align*}
    The estimates of $T^5_2, T^5_3, T^5_4$ are similar to $T^3_1, T^3_2, T^3_2$ respectively, and we omit them here. This finish the improvement of \eqref{bs-w}.
    \section{Appendix}
    This appendix provides the proof of Lemma \ref{lem-stri}.
    \begin{proof}[Proof of Lemma \ref{lem-stri}]
        It suffices to prove the estimates for $\beta=1$. Recalling the linear system \eqref{eq-linear}, the energy estimates gives 
        \begin{align}\label{est-ener}
  		\big|\hat{S}(t,\xi)\hat{F}(\xi)\big|\lesssim e^{-\nu |\xi|^2 t}\big| \hat{F}(\xi)\big|.
  	    \end{align}
    Let $P_j$ denote the Littlewood--Paley multipliers localized to
  	\(|\xi|\simeq R=2^j\). The main step is to prove the bound, 
    \begin{align}\label{est-p}
        \|S(t)P_jF\|_{L^\infty}\lesssim R^2(1+t)^{-\frac{1}{2}} e^{-c\nu R^2t}\|P_jF\|_{L^1},\quad j\in\mathbb{Z}.
    \end{align}
    Bernstein's inequality yields that
  		\begin{equation} \label{est-1}
  		    \begin{aligned}
  			\|S(t)P_jF\|_{L^{\infty}}\lesssim& R\| S(t)P_j F\|_{L^2}\\
  			\lesssim& e^{-c\nu R^2 t}R\|P_j F\|_{L^2}\\
  			\lesssim& e^{-c\nu R^2 t}R^2 \|P_j F\|_{L^1},
  		\end{aligned}
  		\end{equation}
        then it suffices to prove \eqref{est-p} for $t \geqslant1$.
    Let $\alpha_j=|\nu_2|R^2$. Suppose that $\alpha_j\geq c_0$, where $c_0>0$ is sufficient small. Then it holds that $\nu R^2\gtrsim c_1$ for some $c_1>0$. Combining with \eqref{est-1}, we have
  		\begin{align*}
  			\|S(t)P_jF\|_{L^{\infty}}\lesssim&  (1+t)^{-\frac{1}{2}} e^{-c'\nu R^2 t}R^2\|P_j F\|_{L^1},
  		\end{align*}
        for some $c'>0$. We next consider the case $\alpha_j\leq c_0$. Let $\chi_j$ denote the symbol of $P_j$. By Inverse Fourier transformation, it suffices to prove that
  		\begin{equation*}
  			K_j(t,y,z)=\int_{\mathbb{R}^2} e^{i(y\eta+zl)}\hat{S}(t,\xi)\chi_j(\xi) d\xi
  		\end{equation*} 
  		satisfies
  		\begin{align*}
  			\|K_j(t,y,z)\|_{L^\infty_{y,z}}\lesssim R^2 (1+t)^{-\frac{1}{2}} e^{- c\nu R^2 t}.
  		\end{align*}
        We express $K$ in the polar coordinate  $(\eta,l)=Rr(\cos\theta,\sin\theta)$ and set
  		 \begin{align*}
  		 	(y,z)=R^{-1}\rho(\cos\varphi,\sin\varphi),\quad b_j(\theta)=\gamma\left(\frac{\sin\theta}{16\alpha_j}\right),
  		 \end{align*}
  		 where $\gamma\in C_c^{\infty}(\mathbb{R})$ is an even cut-off function satisfying
  		 \begin{equation*}
  		 	0\leq\gamma\leq1,\quad\gamma(x)=1\ (|x|\leq1),\quad \gamma(x)=0\ (|x|\geq2).
  		 \end{equation*} 
  	We decompose $K$ into two parts
  	\begin{align*}
  		&K_j^1(t,\rho,\varphi)=R^2\int_{1/2}^2\int_{-\pi}^{\pi}e^{ir\rho\cos(\theta-\varphi)}\hat{S}(t, r,\theta)b_j(\theta)A(r)r d \theta d r,\\
  		&K_j^2(t,\rho,\varphi)=R^2\int_{1/2}^2\int_{-\pi}^{\pi}e^{ir\rho\cos(\theta-\varphi)}\hat{S}(t,r,\theta)(1-b_j(\theta))A(r)r d \theta d r,
  	\end{align*}
  	where \(A(r)=\chi_0(\xi)\).
    For $K^1_j$, since the angular support of $b_j$ has measure $O(\alpha_j)$, the energy estimate \eqref{est-ener} implies
  	\begin{align*}
  		\|K^1_j\|_{L^\infty_{\rho,\varphi}}\lesssim R^2e^{-c\nu R^2 t}\alpha_j\lesssim R^2(1+t)^{-\frac{1}{2}}e^{-c'\nu R^2 t}.
  	\end{align*}
    Turning to $K^2_j$, it suffices to consider $0\leq\theta\leq\pi$, and the other case can be treated similarly. On the support of $1-b_j$, we have
  	\begin{align}\label{est-k2}
  		\alpha_j r^2\leq 4\alpha_j\leq\frac{1}{4}\sin(\theta).
  	\end{align}
  	Moreover, the representation of $\hat{S}(\xi)$ can be transformed into a sum of two branches
  	\begin{equation*}
  		\hat{S}(t,r,\theta)=e^{-\nu_1 R^2 r^2 t}\sum_{\mathfrak{c}=\pm1}e^{\mathfrak{c} it\sin\theta}e^{ \mathfrak{c}it\psi_1(r,\theta)}\Xi^{\mathfrak{c}}(r,\theta),
  	\end{equation*}
  	where we denote
  	\begin{align*}
    &\psi(r,\theta)=\big((\sin\theta)^2-\alpha_j^2r^4\big)^{\frac{1}{2}}, \quad N(r,\theta)=\left(\begin{array}{cc}
  			\nu_2R^2r^2&i\sin \theta\\
  			i\sin \theta&-\nu_2R^2r^2
  		\end{array}\right),\\
  		&\psi_1(r,\theta)=\psi-\sin\theta,\quad\Xi^{\mathfrak{c}}(r,\theta)=\frac{1}{2}\left(I-\mathfrak{c} i\frac{N}{\psi}\right).
  	\end{align*}
    By symmetry, it suffices to consider $\mathfrak{c}=1$ and we therefore omit the superscript $\mathfrak{c}$ in the following proof.
  	Define
  	\begin{equation*}
  		B(t,\theta,z)=\int_{1/2}^{2}e^{-\nu_1 R^2 r^2 t}e^{irz+it\psi_1(r,\theta)}\Xi(r,\theta)(1-b_j(\theta))A(r)r dr.
  	\end{equation*}
    We claim the following bounds
  	\begin{align}\label{est-B}
  		\|B(\theta,z)\|_{L^\infty_{\theta,z}}+\int_{0}^{\pi}\|\partial_{\theta}B(\theta,z)\|_{L^\infty_{z}}d\theta+\int_{\mathbb{R}}\|\partial_{z}B(\theta,z)\|_{L^\infty_{\theta}}dz\lesssim C e^{-c'\nu_1R^2t}.
  	\end{align}
  	Indeed, by \eqref{est-k2}, for each component of matrices $\Xi$, it holds that
  	\begin{align*}
  		\frac{|\nu_2 R^2r^2|}{\psi}+\frac{|\sin\theta|}{\psi}\lesssim 1,
  	\end{align*} 
    which gives the first bound. Direct differentiation of $\psi_1$ and $\Xi$ gives 
    \begin{align*}
        |\partial_{\theta}\psi_1|\lesssim\frac{\alpha^2_j|\cos\theta|}{(\sin\theta)^2},\quad |\partial_{\theta}\Xi|\lesssim\frac{\alpha_j|\cos\theta|}{(\sin\theta)^2}.
    \end{align*}
    Hence, on the support of  $1-b_j(\theta)$, it holds that
    \begin{align*}
  		&\left| \int_0^{\pi}\sup_{z}|(\partial_{\theta}B)(t,\theta,z)|d\theta\right|\\
  		\lesssim&e^{-2c\nu_1R^2t}\left(\int_{16\alpha_j}^{1}\left(\frac{\alpha_j^2t}{y}+\frac{\alpha_j}{y^2}+\frac{\alpha_j^2}{y^3} \right)dy+1\right)\\
  		\lesssim&e^{-c\nu_1R^2t}.
  	\end{align*}
    \color{black}
  	Turning to the third bound, direct differentiation gives 
  	\begin{align*}
  		|\partial_r^k\psi_1(r,\theta)|\lesssim \alpha_j,\quad |\partial_r^k\Xi^{\pm}(r,\theta)|\lesssim1,\quad \mathrm{for}\ k=1,2.
  	\end{align*}
  	Integrating by parts in $r$ twice for $|z|\geq1$, one have
  	\begin{align*}
  		&\int_{\mathbb{R}}\sup_{\theta}|(\partial_zB)(t,\theta,z)|dz\\
  		\lesssim&\int_{|z|\leq1}\sup_{\theta}|(\partial_zB)(t,\theta,z)|dz+
  		\int_{|z|\geq1}\sup_{\theta}|(\partial_zB)(t,\theta,z)|dz\\
        \lesssim&e^{-c'\nu_1R^2t}+
  		\int_{|z|\geq1}e^{-c'\nu_1R^2t}(1+|z|)^{-2}dz\\
  		\lesssim&  e^{-c'\nu_1R^2t},
  	\end{align*}
    which implies the estimates \eqref{est-B}. 
  	
    Returning to $K^2_j$, we further split it into two parts:
  	\begin{align*}
  		K^{2}_j(t,\rho,\varphi)&=R^2\int_{I_{\delta}}e^{\pm it\sin\theta}B(t,\theta,\rho\cos(\theta-\varphi))d\theta+R^2\int_{[0,\pi]\backslash I_{\delta}}e^{\pm it\sin\theta}B(t,\theta,\rho\cos(\theta-\varphi))d\theta\\
  		&=:K^{2,1}_j+K^{2,2}_j,
  	\end{align*}
  	where $I_{\delta}=[\frac{\pi}{2}-\delta,\frac{\pi}{2}+\delta]$.
  	For $K^{2,1}_j$, 
  	combining the fact $|I_{\delta}|\leq 2\delta$ with \eqref{est-B}, we deduce that
  	\begin{align*}
  		\|K^{2,1}_j\|_{L^\infty_{\rho,\varphi}}\lesssim \delta R^2 e^{-c\nu_1 R^2 t}.
  	\end{align*}
    Turning to $K^{2,2}_j$, it holds on the interval $[0,\pi]\backslash I_{\delta}$ that
  	\begin{equation}\label{est-g}
  		|\cos\theta|\geq \sin\delta\geq\frac{\delta}{2}.
  	\end{equation}
  	Consequently, integration by parts in $\theta$, together with \eqref{est-B}, gives
  	\begin{align*}
  		K^{2,2}_j=&R^2\int_{[0,\pi]\backslash I_{\delta}}\frac{1}{it\cos\theta}\partial_{\theta}\left(e^{\pm it\sin\theta}\right)B(t,\theta,\rho\cos(\theta-\varphi))d\theta\\
  		\lesssim&\frac{R^2}{\delta t}	\|B(\theta,\rho\cos(\theta-\varphi))\|_{L^\infty_{\theta}}\\
  		&+R^2\left| \int_{[0,\pi]\backslash I_{\delta}}\frac{1}{it\cos\theta}e^{\pm it\sin\theta}(\partial_{\theta}B)(t,\theta,\rho\cos(\theta-\varphi))d\theta\right|\\
        &+R^2\left| \int_{[0,\pi]\backslash I_{\delta}}\frac{1}{it\cos\theta}e^{\pm it\sin\theta}(\partial_{z}B)(t,\theta,\rho\cos(\theta-\varphi))\rho\sin(\theta-\varphi)d\theta\right|\\
        \lesssim &\frac{R^2}{\delta t}e^{-c'\nu_1 R^2 t}+\frac{R^2}{\delta t}\int_{\mathbb{R}}|(\partial_zB)(t,\theta(z),z)|dz\\
  		\lesssim &\frac{R^2}{\delta t}e^{-c'\nu_1 R^2 t}+\frac{R^2}{\delta t}\int_{\mathbb{R}}\sup_{\theta}|(\partial_zB)(t,\theta,z)|dz\\
  		\lesssim &\frac{R^2}{\delta t}e^{-c'\nu_1 R^2 t}.
  	\end{align*}
  	Here we change the variable $z=\rho\cos(\theta-\varphi)$ on each interval where $\cos(\theta-\varphi)$ is monotone in the third-to-last inequality.
  	Choosing $\delta=t^{-\frac{1}{2}}$ proves \eqref{est-p}. Summing over $j\in\mathbb{Z}$ then yields \eqref{est-d1}. Indeed, it follow from
    \begin{align*}
  		\|S(t)F\|_{L^\infty}\lesssim& (1+t)^{-\frac{1}{2}}(\nu t)^{-\delta}
  		\left(\sum_{j=0}^{\infty}2^{(2-2\delta)j}\|P_j F\|_{L^1}+\sum_{j=1}^{\infty}2^{-(2-2\delta)j}\|P_{-j} F\|_{L^1}\right)\\
  		\lesssim&(1+t)^{-\frac{1}{2}}(\nu t)^{-\delta}\|F\|_{W^{2,1}}.
  	\end{align*}
    The estimate of \eqref{est-d2} follows from a standard $TT^*$ argument. 
  	Define operator $T$ by the symbol
  	\begin{equation*}
  		\widehat{TF}(t,\xi)=\hat{S}(t,\xi)\chi_j(\xi)\hat{F}(\xi).
  	\end{equation*}
    
  	We decompose $T$ into there multipliers
  	\begin{align*}
    &T=T^{+1}_{1}+T_1^{-1}+T_2+T_3\\
  		&\widehat{T^{\mathfrak{c}}_1F}(t,\xi)=e^{-\nu_1|\xi|^2t+ it\mathfrak{c}\psi(\xi)}\Xi^{\mathfrak{c}}(\xi)(1-\mathfrak{b}_j(\xi))\chi_j(\xi)\mathbf{1}_{\{\alpha_j\leq c_0\}}\hat{F}(\xi),\\
  		&\widehat{T_2F}(t,\xi)=\hat{S}(t,\xi)\,\mathfrak{b}_j(\xi)\chi_j(\xi) \mathbf{1}_{\{\alpha_j\leq c_0\}}\hat{F}(\xi),\\
  		&\widehat{T_3F}(t,\xi)=\hat{S}(t,\xi)\chi_j(\xi) \mathbf{1}_{\{\alpha_j\geq c_0\}}\hat{F}(\xi).
  	\end{align*}
  	where $\mathfrak{b}_j(\xi)=b_j(\theta)$. 
  	For a function $G(t,y,z)\in C_c^{\infty}(\mathbb{R}^+\times\mathbb{R}^2)$, the space-time adjoint $T_1^{\mathfrak{c},*}$ is given by
  	\begin{align*}
  		&\widehat{T_1^{\mathfrak{c},*}G}(\xi)=\int_{0}^{\infty}\hat{T}_1^*(s,\xi)\hat{G}(s,\xi) ds,\quad\hat{T}_1^*(s,\xi)=e^{-\nu_1|\xi|^2s- is\psi(\xi)}\Xi^*(\xi)(1-\mathfrak{b}_j(\xi))\chi_j(\xi).
  	\end{align*}
  	Then it holds that
  	\begin{align*}
  		\|T_1^{\mathfrak{c},*}G\|_{L^2_{y,z}}^2=&\left\langle \int_0^{\infty} T_1^{\mathfrak{c},*}(s)G(s)ds,\int_0^{\infty} T_1^{\mathfrak{c},*}(t)G(t)dt\right\rangle_{L^2_{y,z}}\\
  		=&\int_{0}^{\infty}\int_{0}^{\infty}\left\langle T^{\mathfrak{c}}_1(t)T^{\mathfrak{c},*}_1(s)G(s),G(t)\right\rangle_{L^2_{y,z}}dsdt,
  	\end{align*}
  	where
  	\begin{align*}
  		\hat{T}^{\mathfrak{c}}_1(t)\hat{T}_1^{\mathfrak{c},*}(s)(\xi)=e^{-\nu_1|\xi|^2(t+s)+ i(t-s)\psi(\xi)}\Xi(\xi)\Xi^*(\xi)(1-\mathfrak{b}_j(\xi))^2\chi_j^2(\xi).
  	\end{align*}
  	Arguing as in the estimate of $K^2_j$, one can prove that
  	\begin{align*}
  		\|T^{\mathfrak{c}}_1(t)T_1^{\mathfrak{c},*}(s)G(s)\|_{L^{\infty}_{y,z}}\lesssim R^2e^{-c\nu_1R^2(t+s)}(1+|t-s|)^{-1/2}\|G(s)\|_{L^1_{y,z}}.
  	\end{align*}
  	Let $g(t)=\|G(t)\|_{L^1_{y,z}}$. It follows from Cauchy-Schwarz inequality that
  	\begin{align*}
  		\|T_1^{\mathfrak{c},*}G\|_{L^2_{y,z}}^2\lesssim&\iint\|T_1(t)T_1^{\mathfrak{c},*}(s)G(s)\|_{L^\infty}\|G(t)\|_{L^1}dsdt\\
  		\lesssim&R^2\iint e^{-c\nu_1R^2(t+s)}(1+|t-s|)^{-1/2}g(s)g(t)dsdt\\
  		\lesssim&R^2\iint e^{-c\nu_1R^2(t+s)}(1+|t-s|)^{-1/2}g(t)^2 dsdt.,\\
  		\lesssim&R^2\sup_{t\geq0}\int_{0}^{\infty}e^{-c\nu_1R^2(t+s)}(1+|t-s|)^{-1/2}ds \|g\|_{L^2_t}^2=:R^2 \mathfrak{C} \|G\|_{L^2_tL^1_{y,z}}^2.
  	\end{align*}
  	Since $t+s\geq|t-s|$ for $t,s\geq0$, we obtain
  	\begin{align*}
  		\mathfrak{C}=&\sup_{t\geq0}\int_{0}^{\infty}e^{-c\nu_1R^2(t+s)}(1+|t-s|)^{-1/2}ds\\
  		\leq&2\int_{0}^{\infty}e^{-c\nu_1R^2 \tau}(1+\tau)^{-\frac{1}{2}}d\tau\\
  		\lesssim&\int_{0}^{1}(\nu_1R^2\tau)^{-1/2}d\tau+\int_{1}^{\infty}e^{-c\nu_1R^2\tau}\tau^{-\frac{1}{2}}d\tau\\
  		\lesssim&\nu^{-\frac{1}{2}}R^{-1}.
  	\end{align*}
  	Therefore, it holds that
  	\begin{align*}
  		\|T_1^{\mathfrak{c},*}G\|_{L^2_{y,z}}\lesssim \nu^{-\frac{1}{4}}R^{\frac{1}{2}}\|G\|_{L^2L^1}.
  	\end{align*}
  	We get by duality that
  	\begin{align*}
  		\|T^{\mathfrak{c}}_1F\|_{L^2_tL^\infty_{y,z}}=&\sup_{\|G\|_{L^2_tL^1_{y,z}}\leq1}\left|\int_{0}^{\infty} \langle T^{\mathfrak{c}}_1(t)F,G(t)\rangle dt\right| \\
  		=&\sup_{\|G\|_{L^2_tL^1_{y,z}}\leq1}\left| \langle P_jF,T_1^{\mathfrak{c},*}G\rangle \right| \\
  		\lesssim& \nu^{-\frac{1}{4}}R^{\frac{1}{2}}\|P_jF\|_{L^2_{y,z}}.
  	\end{align*}
  	The operators $T_2$ and $T_3$ satisfy the same bound by the argument used in the proof of \eqref{est-d1}. Summing over $j\in\mathbb{Z}$ yields $\eqref{est-d2}$. Finally, \eqref{est-d3} follows by combining the preceding procedure with the argument of \cite[Lemmas A.3 and A.5]{LSWZ25}. This completes the proof. 
    \end{proof}

    \vspace{4 mm}
	\noindent \textbf{Acknowledgements:} The authors would like to thank Prof. Fei Wang and Prof. Chengjie Liu for the meaningful discussions.


\begin{thebibliography}{99}
  		
		\bibitem{A25}
		R. Arbon, Enhanced Dissipation, Taylor Dispersion, and Inviscid Damping of Couette flow in the Boussinesq system on the Plane. Nonlinearity {\bf 38} (2025), no.~10, Paper No. 105011, 76 pp.
		
		\bibitem{AB25}	
		R. Arbon, J. Bedrossian, Quantitative hydrodynamic stability for Couette flow on unbounded domains with Navier boundary conditions. Comm. Math. Phys. 406(6): Paper No. 129, 57 pp (2025).
		
		\bibitem{BBCD2023}
		J. Bedrossian, R. Bianchini, M. Coti Zelati, M. Dolce, Nonlinear inviscid damping and shear-buoyancy instability in the two-dimensional Boussinesq equations. Comm. Pure Appl. Math. 76(12): 3685–3768 (2023).
  		
  		
  		\bibitem{BGM17}
  		J. Bedrossian, P. Germain, N. Masmoudi, On the stability threshold for the 3D Couette flow in
  		Sobolev regularity. Ann. of Math.  185(2): 541--608 (2017).
  		
  		\bibitem{BGM20}
  		J. Bedrossian, P. Germain, N. Masmoudi, Dynamics near the subcritical transition of the 3D Couette flow I: Below threshold case. Mem. Amer. Math. Soc. 266(1294), v+158pp (2020).
  		
  		\bibitem{BGM22}
  		J. Bedrossian, P. Germain, N. Masmoudi, Dynamics near the subcritical transition of the 3D Couette flow II: Above threshold case. Mem. Amer. Math. Soc. 279(1377) v+135pp (2022).
  		
  		\bibitem{BHILW25a}  J. Bedrossian, S. He, S. Iyer, L. Li, and F. Wang, Stability threshold of close-to-Couette shear flows with no-slip boundary conditions in 2D,
  		\textit{arXiv preprint arXiv:2510.16378}, 2025.
  		
  		\bibitem{BHIW24}
  		J. Bedrossian, S. He, S. Iyer, F. Wang, Uniform Inviscid Damping and Inviscid Limit of the 2D Navier-Stokes equation with Navier Boundary Conditions. arXiv:2405.19249  (2024).
  		
  		\bibitem{BHIW25}
  		J. Bedrossian, S. He, S. Iyer, F. Wang, Stability threshold of nearly-Couette shear flows with Navier boundary conditions in 2D.  Comm. Math. Phys. 406(2): Paper No. 28, 42 pp (2025).
 
  		\bibitem{BMV16}
  		J. Bedrossian, N. Masmoudi, V. Vicol, Enhanced dissipation and inviscid damping in the inviscid limit of the Navier-Stokes equations near the two dimensional Couette flow. Arch. Ration. Mech. Anal. 219(3): 1087--1159  (2016).
  		
  		\bibitem{BVW18}
  		J. Bedrossian, V. Vicol, F. Wang, The Sobolev stability threshold for 2D shear flows near Couette. J. Nonlinear Sci. 28(6): 2051--2075 (2018).
  	
  		
  		
  		\bibitem{CLWZ20}
  		Q. Chen, T. Li, D. Wei, Z. Zhang, Transition threshold for the 2-D Couette flow in a finite channel. Arch. Ration. Mech. Anal. 238(1): 125--183 (2020).
  		
  		\bibitem{CWZ24}
  		Q. Chen, D. Wei, Z. Zhang, Transition threshold for the 3D Couette flow in a finite channel. Mem. Amer. Math. Soc. 296(1478), v+178pp (2024).
  		
  		\bibitem{CD23}
  		M. Coti Zelati, A. Del Zotto, Suppression of lift-up effect in the 3D Boussinesq equations around a stably stratified Couette flow. Quarterly of Applied Mathematics, 83(2): 389–401 (2024).
  		
  		\bibitem{CDW24}
  		M. Coti Zelati, A. Del Zotto, K. Widmayer, Stability of viscous three-dimensional stratified Couette flow via dispersion and mixing. arXiv:2402.15312 (2024).
  		Comm. Pure Appl. Math. (2026): e70048. https://doi.org/10.1002/cpa.70048
  		
  		
  		\bibitem{CDW25}
  		M. Coti Zelati, A. Del Zotto, K. Widmayer, On the stability of viscous three-dimensional rotating Couette flow. arXiv:2501.17735 (2025).	
  		
  		\bibitem{C23}
  		Q. Chen, S. Ding, Z. Lin, and Z. Zhang. Nonlinear stability for 3-d plane poiseuille flow in a
  		finite channel. arXiv preprint arXiv:2310.11694, 2023.
  		
        \bibitem{CWY25}
        Y. Chen, W. Wang, G. Yang, Stability threshold of Couette flow for Boussinesq equations in $\mathbb{R}^2$. 	arXiv:2508.11908 (2025).
  		
  		\bibitem{CEW20} M. Coti Zelati, T. M. Elgindi, and K. Widmayer, Enhanced dissipation in the Navier-Stokes equations near the Poiseuille flow, Comm. Math. Phys., 378, 2020, 987-1010.
  		
  		\bibitem{CJWZ25}
  		Q. Chen, H. Jia, D. Wei, and Z. Zhang. Asymptotic stability of the Kolmogorov flow at high	reynolds numbers. arXiv preprint arXiv:2510.13181, 2025.
  		
  		\bibitem{CWW25}
  		S. Cui, L. Wang, W. Wang, Stability threshold of Couette flow for 3D Boussinesq system in Sobolev spaces. arXiv:2504.16401 (2025).
  		
  		\bibitem{DWZ21}
  		W. Deng, J. Wu, P. Zhang, Stability of Couette flow for 2D Boussinesq system with vertical dissipation. J. Funct. Anal. 281(12): Paper No. 109255, 40 pp (2021).
  		
  		
  		\bibitem{HNZ26}
  		S. He, B. Niu, W. Zhao, Couette flow with Robin boundary condition (I): the viscosity-independent friction. arXiv:2607.17468 (2026).
  		
  		\bibitem{HLX26} 
  		F. Huang, R. Li and L. Xu,
  		Nonlinear stability threshold for compressible Couette flow,
  		\textit{Comm. Math. Phys.}, 407 (1), 11 (2026).
  		  		
  		\bibitem{HLSX26}
  		W. Huang, Z. Luo, Y. Sun, X. Xu, Stability threshold for 3D Boussinesq equations with rotation near the Couette flow and stratified temperature. arXiv:2602.10591 (2025).
  		
  		\bibitem{HSX24a}
  		W. Huang, Y. Sun, X. Xu, On the Sobolev stability threshold for 3D Navier-Stokes equations with rotation near the Couette flow. arXiv:2409.05104 (2024).
  		
  		
  		
  		\bibitem{K25} N. Knobel, Suppression of fluid echoes and Sobolev stability threshold for 2D dissipative fluid equations around Couette flow, arXiv:2505.23391v1.
  		
  		
  		\bibitem{LLZ25}
  		H. Li, N. Liu, W. Zhao, Stability threshold of the two-dimensional Couette flow in the whole plane. arXiv:2501.10818 (2025).
  		
  		\bibitem{LMZ25}
  		H. Li, N. Masmoudi, W. Zhao, Asymptotic Stability of Two-Dimensional Couette Flow in a Viscous Fluid, Arch. Ration. Mech. Anal. 249(5): Paper No. 56 (2025).
  		
  		
  		\bibitem{LSWZ25}
  		M. Li, C. Sun, C. Wang, D. Wei, Z. Zhang, Transition threshold for the Navier-Stokes-Coriolis system at high Reynolds numbers. arXiv:2511.21294v1 (2025).
  		
  		
  		\bibitem{LWZ25} M. Li, C. Wang, and Z. Zhang, Nonlinear stability of 2-D Couette flow for the compressible Navier-Stokes equations at high Reynolds number.   \textit{ArXiv}: 2508.07291, (2025). 
  		
  		
  		\bibitem{LWZ20} T. Li, D. Wei and Z. Zhang, Pseudospectral bound and transition threshold for the 3D Kolmogorov flow, Comm. Pure Appl. Math., 73, 2020, 465-557.  		
  		
  		\bibitem{LSZ26} Z. Li, S, Shen, Z. Zhang, Asymptotic stability threshold of the 2-D monotone shear flow with no-slip boundary condition. arXiv:2603.01797 (2026).


        \bibitem{LWZ25}
        T. Liang, Y. Li, X. Zhai, Stability threshold for two-dimensional Boussinesq system near-Couette shear flow in a finite channel, arXiv:2504.02669. (2025).
        
  		\bibitem{L20}
  		K. Liss, On the Sobolev stability threshold of 3D Couette flow in a uniform magnetic field. Comm. Math. Phys. 377(2): 859–908 (2020).
  		
  		\bibitem{MSZ22} N. Masmoudi, B. Said-Houari, and W. Zhao, Stability of Couette flow for 2D Boussinesq system without thermal diffusivity, Arch. Ration. Mech. Anal.   245(2), 2022, 645-752 (2022).
  		
  		\bibitem{MZZ23} N. Masmoudi, C. Zhai and W. Zhao, Asymptotic stability for two-dimensional Boussinesq systems around the Couette flow in a finite channel, J. Funct. Anal., 284(1), 2023, 109736.
  		
  		\bibitem{MZ22} N. Masmoudi and W. Zhao, Stability threshold of two-dimensional Couette flow in Sobolev spaces, Ann. Inst. Henri Poincare, 39, 2022, 245-325.
  		
  		
  		
  		
  		\bibitem{NZ24}
  		B. Niu, W. Zhao, Improved stability threshold of the Two-Dimensional Couette flow for Navier-Stokes-Boussinesq Systems via quasi-linearization.  Commun. Pure Appl. Anal, (2024).
  		
		
		\bibitem{RW26}
		X. Ren and D. Wei, Transition threshold of Couette flow for 2D Boussinesq equations, J. Funct. Anal. {\bf 290} (2026), no.~9, Paper No. 111383, 44 pp.
		
  		\bibitem{WZ21}
  		D. Wei, Z. Zhang, Transition threshold for the 3D Couette flow in Sobolev space. Comm. Pure Appl. Math. 74(11): 2398--2479  (2021).
  		
  		\bibitem{WZ23}
  		D. Wei, Z. Zhang, Nonlinear enhanced dissipation and inviscid damping for the 2D Couette flow. Tunis. J. Math. 5(3): 573–592 (2023).
  		
  		\bibitem{WZ26}
  		D. Wei and Z. Zhang, Asymptotic stability threshold of the 2D Couette flow in a finite channel, Ann. Inst. H. Poincar$\mathrm{\acute{e}}$  C Anal. Non Lin$\mathrm{\acute{e}}$aire {\bf 43} (2026), no.~3, 501--540.

        \bibitem{YL18}
        Yang, J., Lin, Z. Linear Inviscid Damping for Couette Flow in Stratified Fluid. J. Math. Fluid Mech. 20, 445–472 (2018).
        
  		\bibitem{ZZ23}
  		C. Zhai, W. Zhao, Stability threshold of the Couette flow for Navier-Stokes Boussinesq system with large Richardson Number $\gamma^2 > \frac 14$. SIAM J. Math. Anal. 55(2): 1284–1318 (2023).
  		
  		\bibitem{ZZi23}
  		Z. Zhang, R. Zi, Stability threshold of Couette flow for 2D Boussinesq equations in Sobolev spaces. J. Math. Pures Appl. 179(9): 123–182 (2023).
  		
  	
  		
  		
  		\bibitem{ZZ25}
  		W. Zhao, R. Zi, Asymptotic stability of Couette flow in a strong uniform magnetic field for the Euler-MHD system. Arch. Ration. Mech. Anal. 248(3):47, 2024.
  		
  		
 		

        
  	\end{thebibliography}
\end{document}